\documentclass{amsart}

\usepackage{amsmath}
\usepackage{amssymb}
\usepackage{graphicx}
\usepackage{tikz}
\usepackage{color}
\usepackage[all]{xy}
\usepackage{amsthm}
\usepackage{verbatim}
\usepackage[normalem]{ulem}

\usetikzlibrary{arrows,shapes,positioning}
\usetikzlibrary{decorations.markings}
\tikzstyle arrowstyle=[scale=1]
\tikzstyle directed=[postaction={decorate,decoration={markings,
    mark=at position .55 with {\arrow[arrowstyle]{stealth}}}}]
\tikzstyle ddirected=[postaction={decorate,decoration={markings,
    mark=at position .45 with {\arrow[arrowstyle]{stealth}},
    mark=at position .55 with {\arrow[arrowstyle]{stealth}}}}]
\tikzstyle reverse directed=[postaction={decorate,decoration={markings,
    mark=at position .55 with {\arrowreversed[arrowstyle]{stealth};}}}]
\tikzstyle reverse ddirected=[postaction={decorate,decoration={markings,
    mark=at position .55 with {\arrowreversed[arrowstyle]{stealth};},
    mark=at position .65 with {\arrowreversed[arrowstyle]{stealth};}}}]

\usepackage[colorlinks,linkcolor=blue,citecolor=blue,pdfstartview=FitH]{hyperref}

\theoremstyle{plain}

\newtheorem{para}{}[section]

\newtheorem{prop}[para]{Proposition}
\newtheorem{lemma}[para]{Lemma}
\newtheorem{cor}[para]{Corollary}
\newtheorem{thm_star}[para]{Theorem*}
\newtheorem{prop_star}[para]{Proposition*}

\newtheorem{fact}[para]{Fact}
\newtheorem*{factstar}{Fact}

\theoremstyle{remark}

\newtheorem*{remark}{Remark}

\theoremstyle{definition}
\newtheorem*{dfnstar}{Definition}
\newtheorem{dfn}[para]{Definition}
\newtheorem{example}[para]{Example}

\newtheorem{conjecture}[para]{Conjecture}

\newcommand{\co}{\colon\thinspace}

\newcommand{\bound}{\partial}

\newcommand{\C}{\mathbb{C}}
\renewcommand{\H}{\mathbb{H}}
\newcommand{\Q}{\mathbb{Q}}
\newcommand{\R}{\mathbb{R}}
\newcommand{\Z}{\mathbb{Z}}

\newcommand\bu{\mathbf{u}}

\newcommand\bz{\mathbf{z}}

\newcommand\fg{\mathfrak{g}}
\newcommand\fl{\mathfrak{l}}
\newcommand\fm{\mathfrak{m}}

\begin{document}

\title{Generic hyperbolic knot complements without hidden symmetries}

\author{Eric Chesebro}
\address{Department of Mathematical Sciences, University of Montana} 
\email{Eric.Chesebro@mso.umt.edu} 

\author{Jason DeBlois}
\address{Department of Mathematics, University of Pittsburgh} \email{jdeblois@pitt.edu}

\author{Priyadip Mondal}
\address{Department of Mathematics, University of Pittsburgh} \email{prm50@pitt.edu}

\begin{abstract}  We establish a pair of criteria for proving that most knots obtained by Dehn surgery on a given two-component hyperbolic link lack hidden symmetries.  To do this, we use certain rational functions on varieties associated to the link. We apply our criteria to show that among certain infinite families of knot complements, all but finitely many members lack hidden symmetries.
\end{abstract}


\maketitle

A longstanding question in the study of hyperbolic 3-manifolds asks which hyperbolic \emph{knot complements}, the 3-manifolds obtained by removing a knot from $S^3$, have hidden symmetries \cite[p.~307]{NeumReid}.  More recent work of Reid--Walsh \cite{ReidWalsh} and Boileau--Boyer--Cebanu--Walsh \cite{BBoCWaGT} relates this to \cite[Conjecture 5.2]{ReidWalsh}, on commensurability classes of knot complements.  We find the original question intriguing simply because hyperbolic $3$-manifolds with hidden symmetries are quite common --- each manifold that non-normally covers another has them --- but hyperbolic knot complements with hidden symmetries seem quite rare.  In fact only three are known to have hidden symmetries, a great many are known not to, and no new examples have been found since the publication of \cite{NeumReid}. Indeed, the authors Neumann and Reid of \cite{NeumReid} later conjectured that no hyperbolic knot complement in $S^3$ has hidden symmetries, beyond the three already known \cite[Problem 3.64(A)]{KirbyList}. 

The totality of evidence for this conjecture would still seem to allow for reasonable doubt.  Hidden symmetries can be ruled out for (almost) any \textit{particular} knot complement by straightforward computations using SnapPy \cite{SnapPy} and Sage \cite{sagemath} (or Snap, see \cite{CGHN}).  For instance, amongst the $300,000$-odd knot complements with at most $15$ crossings, only that of the figure-$8$ has hidden symmetries.  Existing tools are harder to apply to \textit{families} of knot complements and we only know the following classes to lack hidden symmetries: the two-bridge knots other than the figure-$8$ \cite{ReidWalsh}; the $(-2,3,n)$-pretzels \cite{MacMat}; knots obtained from surgery on the Berge manifold \cite{Hoffman_3comm}, \cite{Hoffman_hidden}; and certain highly twisted pretzel knots with at least five twist regions \cite[Prop.~7.5]{Millichap}\footnote{Results of the recent preprint \cite{HMW} treat a much broader class of highly twisted knots.}.  Works of Hoffman \cite{Hoffman}, Boileau--Boyer--Cebanu--Walsh \cite{BoiBoCWa2}, and Millichap--Worden \cite{MilliWord} also bear on the question from other directions.  

Here we provide a plethora of new classes by giving a method for quickly showing that the {\em generic} member of certain families of knot complements produced by hyperbolic Dehn surgery lacks hidden symmetries. We were partly inspired to this by \cite[Prop.~7.5]{Millichap} and the proof of \cite[Theorem 1.1]{Hoffman_3comm}, which are more specific results in the same direction, but we develop a new tool based on the following motto: 

\begin{quote} The cusp parameters of knot complements obtained from a given hyperbolic link complement $M$ by hyperbolic Dehn filling are recorded by rational functions on the character or deformation variety of $M$ that are smooth near the complete structure.\end{quote}

\noindent Section \ref{generic proof} reviews the varieties and states this precisely, in Propositions \ref{regfun} and \ref{ratfun}. 

This motto is not surprising and has a proof in the vein of Neumann--Zagier's seminal work \cite{NZ} where they showed that \textit{analytic} functions on open subsets of the deformation variety record the cusp parameters of Dehn fillings, cf. \cite[Theorem 4.1]{NR}.  Promoting analyticity to rationality facilitates a global perspective on these functions that we exploit in ~Section \ref{sixtwotwo}.  Note that the rationality of the cusp parameter does have a precedent in the literature, see \cite{NeumReid} (cf.~Section \ref{whitehead}) where this is shown in the special case of the Whitehead link complement. 

Our main contribution is to use our motto to connect the hidden symmetries question to geometric isolation phenomena. A key result for this is Proposition 9.1 of \cite{NeumReid}, which connects hidden symmetries of a  hyperbolic knot complement to the geometry of its cusp.  Recall that a \textit{hidden symmetry}  of a space $M$ is a homeomorphism between finite-degree covers of $M$ that does not descend to $M$.  Most work on hidden symmetries of knot complements, including ours, does not directly use this definition. Instead the notion of a \textit{rigid cusp} has become fundamental.

We will say that the \textit{shape} of a cusp of a complete hyperbolic $3$-orbifold is the Euclidean similarity class of a horospherical cross-section, that such a similarity class is \textit{rigid} if it is represented by a quotient of $\R^2$ by a Euclidean triangle group or its index-two orientation preserving subgroup, and that the cusp is \textit{rigid} if its shape is. By \cite[Prop.~9.1]{NeumReid}, a hyperbolic knot complement with hidden symmetries covers an orbifold with a rigid cusp. The following result is our main technical tool.

\newcommand\TwoCptCor{
Suppose that $L=K \sqcup K' \subset S^3$ is a hyperbolic link with components $K$ and $K'$.  If infinitely many manifolds or orbifolds obtained from $S^3-L$ by Dehn filling the $K'$-cusp cover orbifolds with rigid cusps, then 
\begin{enumerate}
	\item the shape of the $K$-cusp of $S^3-L$ covers a rigid Euclidean orbifold, and 
	\item the $K$-cusp is geometrically isolated from the $K'$-cusp. \end{enumerate}
}

\theoremstyle{plain}
\newtheorem*{TwoCpt corollary}{Corollary \ref{TwoCpt}}
\begin{TwoCpt corollary}\TwoCptCor\end{TwoCpt corollary}

For some context, remember that the knot complement theorem \cite[Theorem 2]{GL} implies that more than one Dehn surgery on a component $K'$ of $L = K\sqcup K'\subset S^3$ yields a knot in $S^3$ only if $K'$ is unknotted.  Moreover, when this is the situation, there are infinitely many.  Indeed, if we take $\mu$ to be a meridian for $K'$ and $\lambda$ to be the peripheral curve corresponding to the boundary of an embedded disk in $S^3-K'$, then for any $n\in\mathbb{Z}$ the $\frac1n$ surgery along $K'$ determined by the slope $\mu+n\lambda$ yields a knot $K_n$ in $S^3$.  See, for example \cite[Ch.~9.H]{Rolfsen}.  Furthermore, if $S^3-L$ is hyperbolic then by the hyperbolic Dehn surgery theorem \cite[Th.~5.8.2]{Th_notes} (cf.~eg.~\cite{NZ}, \cite{PePor}), $S^3-K_n$ is hyperbolic for all but finitely many $n$.    

Criterion (1) of Corollary \ref{TwoCpt} easily translates to a condition that can be numerically checked, and in Section \ref{two component} we use SnapPy and Sage to apply it to the census of two-component links in $S^3$ with crossing number at most $9$, tabulated in Appendix C of \cite{Rolfsen}.  Each of these has at least one unknotted component.

\newcommand\DataThm{Let $L$ be a hyperbolic two-component link in $S^3$ with crossing number at most nine. 
At most finitely many hyperbolic knot complements obtained by Dehn filling one cusp of $S^3-L$ have hidden symmetries.}

\theoremstyle{plain}
\newtheorem*{Data theorem}{Theorem* \ref{data}}
\begin{Data theorem}\DataThm\end{Data theorem}

This result is analogous to the computation for knots up to $15$ crossings that we mentioned earlier, in the sense that certain conditions are checked case by case by computer. We have given it an asterisk to indicate that it was established by non-verified computation. It is possible in principle to prove this rigorously by verified computation, starting with HIKMOT \cite{HIKMOT} or H.~Moser's work \cite{Moser}. But our main intent with this result is instead to illustrate a computational method for establishing, informally but with high confidence, that surgery on a given link will generically not yield knots with hidden symmetries.

We discuss Theorem* \ref{data} in Section \ref{two component}, but first consider two special cases: the Whitehead link $5^2_1$, in Section \ref{whitehead}, and $6^2_2$ in Section \ref{sixtwotwo}. These are exceptional in that they do satisfy condition (1) of Corollary \ref{TwoCpt}. In fact their complements are arithmetic and cover the Bianchi orbifolds $\mathbb{H}^3/\mathrm{PSL}(2,\mathcal{O}_1)$ and $\mathbb{H}^3/\mathrm{PSL}(2,\mathcal{O}_3)$, respectively, which each have rigid cusps. But it follows from \cite{ReidWalsh} that exactly one surgery on a component of $5^2_1$ yields a knot (the figure-eight) whose complement has hidden symmetries, and concerning $6^2_2$ we have:

\newcommand\CorEEE{At most finitely many knot complements in $S^3$ obtained by Dehn filling one cusp of $S^3-6^2_2$ have hidden symmetries.}

\newtheorem*{cor eree}{Corollary \ref{eree}}
\begin{cor eree}\CorEEE\end{cor eree}

We prove Corollary \ref{eree} in Section \ref{sixtwotwo} by describing the deformation variety of $S^3-6^2_2$ and its cusp parameter function, using this to show that the cusps are not geometrically isolated from each other, then appealing to condition (2) of Corollary \ref{TwoCpt}. Section \ref{whitehead} gives a similar, mainly expository, treatment of $S^3-5^2_1$ that draws on existing literature including \cite{NeumReid} and \cite{HLM}.

Here we say that a cusp $c$ of a two-cusped hyperbolic manifold $M$ is \textit{geometrically isolated} from the other, $c'$, if the shape of $c$ changes under at most finitely many hyperbolic Dehn fillings of $c'$. This varies a bit from the original definition of Neumann--Reid \cite{NR}, see the Remark below Corollary \ref{TwoCpt}. In terms of our motto, if $M$ has two cusps it is equivalent to the function measuring the parameter of $c$ being constant on the curve in the character or deformation variety of $M$ containing (almost) all hyperbolic structures where $c$ remains complete, see the proof of Theorem 4.2 of \cite{NR}. The proof of Corollary \ref{TwoCpt}(2) exploits this fact in a similar way to \S 1.2 of D.~Calegari's study of geometric isolation \cite{Calegari}.

Beyond $5^2_1$ and $6^2_2$ only three two-component links with at most nine crossings, which all have complements isometric to that of $5^2_1$, satisfy condition (1) of Corollary \ref{TwoCpt}.  So Theorem* \ref{data} follows from the computer check and the results above.

The second main result of Section \ref{two component}  applies the orbifold surgery conclusion of Corollary \ref{TwoCpt} (1) to certain knots obtained as branched covers over a fixed link.

\newcommand\PMPY{There is a family $\{M_n\}_{n\geq 3}$ of non-AP hyperbolic knot complements such that $\mathrm{vol}\,M_n\to\infty$ as $n\to\infty$ and $M_n$ lacks hidden symmetries for all but finitely many $n$.}

\newtheorem*{pmpy prop}{Proposition* \ref{pmpy}}
\begin{pmpy prop}\PMPY\end{pmpy prop}

A knot $K$ is \textit{AP}, short for {\em accidental parabolic}, if every closed incompressible surface $S\subset S^3-\mathcal{N}(K)$, where $\mathcal{N}(K)$ is a regular neighborhood of $K$, contains an essential closed curve that bounds an annulus immersed in $S^3-\mathcal{N}(K)$ with its other boundary component on $\partial \mathcal{N}(K)$. The class of AP knots vacuously contains all small knots, since these have no closed incompressible surfaces in their complements, and many other substantial and well-studied classes.   For instance the alternating \cite{Menasco} and Montesinos \cite{Adams_toroidal} knots are all AP.

The non-AP knots have emerged as especially inscrutable from the standpoint of hidden symmetries. Indeed, no infinite family of non-AP knot complements appears to have been previously known to lack them. And the main result of \cite{BoiBoCWa2} significantly constrains the possible hidden symmetries on any AP knot complement, suggesting that \textit{non}-AP knot complements may be the best places to look for  hidden symmetries. Proposition* \ref{pmpy} is of special interest in this regard.

Section \ref{examples} applies our methods to other classes of examples, re-proving the generic case of some of the previous results that we listed above: the $(-2,3,n)$-pretzel knots in Example \ref{2 3 n} (cf.~\cite{MacMat}), and the Berge manifold in Example \ref{protoBerge} (cf.~\cite{Hoffman_3comm}, \cite{Hoffman_hidden}).  We also combine our results with existing work of Aaber--Dunfield \cite{AabD} to address the $(-2,3,8)$ pretzel link, in Example \ref{potential}. 

\begin{remark}  The Whitehead link, the $(-2,3,8)$-pretzel link, the Berge manifold and the link $6^2_2$ are the four link complements addressed in this paper satisfy condition (1) of Corollary \ref{TwoCpt} and they all fail Corollary \ref{TwoCpt}(2). It seems interesting that these are exactly the exceptional examples listed in \cite[Table 7]{MartPet}. As observed there, these are the four two-cusped manifolds with least known volume, and the four with least possible (Matveev) complexity $4$. They are also all arithmetic, the former two covering $\mathbb{H}^3/\mathrm{PSL}_2(\mathcal{O}_1)$ and the latter two $\mathbb{H}^3/\mathrm{PSL}_2(\mathcal{O}_3)$.\end{remark}


Our results and methods support the following conjecture.

\begin{conjecture}\label{more generic}  For any $R>0$, at most finitely hyperbolic knot complements have hidden symmetries and volume less than $R$.\end{conjecture}

This is more modest than Neumann-Reid's conjecture \cite[Problem 3.64(A)]{KirbyList} (cf.~\cite[Conj. 1.1]{BoiBoCWa2}). But it admits a reformulation in terms of Dehn surgery using the ``J\o rgensen--Thurston theory'', and we hope that it may prove more approachable than the original. 

\subsection*{Acknowledgements} We are very grateful to Nathan Dunfield for assistance with SnapPy and Sage, and we thank Nathan and Neil Hoffman for pointing out their relevant works \cite{AabD} and \cite{Hoffman_3comm} (cf.~Examples \ref{potential} and  \ref{protoBerge}).

\section{The cusp parameter as a rational function}\label{generic proof}

Here we will establish our main technical tool Corollary \ref{TwoCpt} for showing that knots obtained by Dehn surgery lack hidden symmetries. Like most other work on this subject, ours exploits the following fundamental characterization due to Neumann--Reid \cite[Proposition 9.1]{NeumReid}: a hyperbolic knot complement in $S^3$ has a hidden symmetry if and only if it covers an orbifold with a rigid cusp. 


We prove the Corollary by combining the next lemma with the fact that the cusp shapes of orbifolds obtained by Dehn filling a fixed cusped hyperbolic manifold $N$ are tracked by a rational function on a variety associated to $N$. This function takes a different form depending on the variety considered, and we will explore it on the character variety in Section \ref{charvar} and the deformation variety in Section \ref{defvar}.

\begin{lemma} \label{lem: finite image}
Suppose for some $B>0$ that $\{M_j\}$ is a collection of complete, one-cusped hyperbolic 3-orbifolds, each with volume at most $B$, such that for every $j$ there is an orbifold cover $M_j\to O_j$ to an orbifold with a rigid cusp.  Then among all $M_j$ there are only finitely many cusp shapes.\end{lemma}

\begin{proof}  For each $j$ the branched cover $M_j\to O_j$ restricts on any horospherical cusp cross-section of $M_j$ to a branched cover of a horospherical cusp cross-section of $O_j$.  It follows that the lattice in $\mathbb{R}^2$ that uniformizes the cusp shape of $M_j$ is a subgroup of the uniformizing lattice for the cusp shape of $O_j$, so of the $(2,3,3)$, $(2,4,4)$, or $(2,3,6)$-triangle group.  The index of this subgroup is either the degree $d_j$ of $M_j\to O_j$ or $2d_j$, depending on whether the cusp cross-section of $O_j$ is a triangle orbifold or turnover.

Let $V$ be the minimal volume of complete hyperbolic $3$-orbifolds \cite{M}.  Then
$$ d_j \leq \frac{B}{V}. $$
The lemma now follows from the basic fact that any finitely generated group, so in particular each of the $(2,3,6)$-, $(2,4,4)$-, and $(3,3,3)$-triangle groups, has only finitely many subgroups of index smaller than a fixed constant.\end{proof}

We now define the invariant of Euclidean tori that we will use to extract information from Lemma \ref{lem: finite image}.

\begin{dfnstar} Suppose $T=\C/\Lambda$ is an oriented Euclidean torus, where $\Lambda \subset \C$ is a lattice, and fix an oriented pair of generators $\mu, \lambda$ for $\Lambda$. The {\it complex modulus of $T$ relative to $(\mu,\lambda)$} is the ratio $\lambda/\mu$ in the upper half-plane $\mathbb{H}^2$. The {\it complex modulus} of $T$ is the orbit of $\lambda/\mu$ under the $\mathrm{PSL}_2(\mathbb{Z})$-action on $\mathbb{H}^2$ by M\"obius transformations, or equivalently, the projection of $\lambda/\mu$ to the modular orbifold $\mathbb{H}^2/\mathrm{PSL}_2(\mathbb{Z})$.
\end{dfnstar}

The definition is motivated by the fact that changing the choice of generating pair changes the relative complex modulus by the action of a M\"obius transformation. By design, the complex modulus does not depend on the choice of a generating set. In fact, complex modulus is an orientation-preserving similarity invariant of $T$.  A similarity between a pair of Euclidean tori lifts to a similarity of $\mathbb{C}$ of the form $z\mapsto \eta z + \tau$ for some $\eta, \tau \in \C$. Conjugating $\Lambda$ by this map in the similarity group of $\mathbb{C}$ has the effect of multiplying all its elements by $\eta$. 

Furthermore, if $\Lambda = \langle\mu,\lambda\rangle$ then conjugating by $z\mapsto (1/\mu) z$ normalizes $\Lambda$, taking $\mu$ to $1$ and $\lambda$ to the complex modulus relative to $(\mu,\lambda)$. We thus have:

\begin{factstar} If $M$ is a cusped hyperbolic 3-orbifold and an oriented Euclidean torus is similar to a cross-section of a cusp of $M$ then the complex modulus of the torus is a complete invariant of the shape of the cusp.\end{factstar}

\begin{dfnstar} For $M$ as above, we refer to the complex modulus of a cusp cross-section as the {\it cusp parameter} of the cusp.\end{dfnstar}

In the subsections that follow, we interpret this basic material in the context of varieties associated to hyperbolic $3$-manifolds.

\subsection{From the perspective of the character variety}\label{charvar} 

The use of representation and character varieties to study three-manifolds goes back at least to Thurston's notes \cite[Chapters 4 \& 5]{Th_notes}. As defined in Section 1 of \cite{CuSh} or Section 4.1 of \cite{Sh}, the $\text{SL}_2 (\C)$-representation variety $R(\Pi)$ of a finitely generated group $\Pi$ is the complex affine algebraic set of representations $\Pi \to \text{SL}_2 (\C)$.  The {\it character} of $\rho \in R(\pi)$, 
\[ \chi_\rho \co \Pi \to \C,\] 
is the function which takes $\gamma \in \Pi$ to $\text{Tr}(\rho(\gamma))$.  As explained in Prop. 1.4.4 and Cor. 1.4.5 of \cite{CuSh}, the set of characters of elements in $R(\Pi)$ is also parametrized by a complex affine algebraic set $X(\Pi)$ and there is a surjective, regular map 
\[ t \co R(\Pi) \to X(\Pi) \] 
which takes representations to their characters. Both $R(\Pi)$ and $X(\Pi)$ are defined over $\Q$.  For each $\gamma \in \Pi$, we have a {\it distinguished trace function} $I_\gamma$ in the coordinate ring $\C[X(\Pi)]$ defined by $I_\gamma(\chi)=\chi(\gamma)$.  Usually, the regular function $I_\gamma \circ t \in \C[R(\Pi)]$ is also denoted as $I_\gamma$. 

It is helpful to notice that an epimorphism $\phi \co \Pi_1 \to \Pi_2$ between finitely generated groups gives a natural inclusion $X(\Pi_2) \to X(\Pi_1)$ given by $\chi \mapsto \chi \circ \phi$.  This inclusion is an affine isomorphism onto its image, so we usually view $X(\Pi_2)$ as a subset of $X(\Pi_1)$.  In particular, when a 3-manifold $M'$ is obtained from a 3-manifold $M$ by Dehn filling of a torus boundary component, we view $X(\pi_1 M')$ as a subset of $X(\pi_1 M)$.

Suppose now that $M$ is a complete hyperbolic manifold $M$ with finite volume.  If $\chi \in X(\pi_1 M)$ is a character of a discrete, faithful representation it is called a {\it canonical character} for $M$.  Prop. 3.1.1 of \cite{CuSh} (attributed there to Thurston) states that $M$ has at least one canonical character.  

If $X$ is an irreducible algebraic component of $X(\pi_1 M)$ which contains a canonical character then $X$ is called a {\it canonical component} of $X(\pi_1 M)$.  Theorem 4.5.1 of \cite{Sh} shows that the dimension of any canonical component for $M$ is equal to the number of cusps of $M$.  This theorem also states that if $\{ \gamma_j \}_1^n$ is a collection of non-trivial peripheral elements of $\pi_1 M$, each of which is carried by a distinct cusp of $M$, then the canonical characters in $X$ are isolated in the subvariety determined by the equations $\{ I_{\gamma_j}^2-4 \}_1^n$.

Next, we state a version of Thurston's Hyperbolic Dehn Surgery Theorem based on that stated in Section 4.11 of \cite{Mar}. For the extension to orbifolds, see eg.~\cite{DM}.  For a cusp $c$ of a $3$-manifold and a fixed generating pair $(\lambda,\mu)$ for $H_1(c)$, the orbifold Dehn fillings of $c$ are parametrized by a pair $(m,n) \in\mathbb{Z}^2$ of \textit{filling coefficients}.  If $r$ is the greatest common factor of $m$ and $n$ then, in the filled orbifold, a simple closed curve representing $\frac{m}{r} \, \lambda + \frac{n}{r} \, \mu$ bounds a cone-disk with an angle of $\frac{2\pi}{r}$ at its center.  The $\infty${\it-surgery} refers to an unfilled cusp, and a sequence of filling coefficients $(m_k,n_k)$ \textit{approaches $\infty$} if it escapes compact sets as $k\to\infty$. (This does not depend on the choice of generating pair.)

\theoremstyle{plain}
\newtheorem*{DSthrm}{Hyperbolic Dehn Surgery Theorem} 
\begin{DSthrm} 
Suppose that $M$ is a non-compact complete hyperbolic orbifold with finite volume.  Assume that $\{M_j\}_1^\infty$ is a sequence of 3-orbifolds obtained from $M$ by distinct Dehn fillings of $M$ whose filling coefficients approach $(\infty,\hdots,\infty)$.  
\begin{enumerate}
\item There is a positive integer $N$ such that if $j \geq N$ then $M_j$ admits a complete hyperbolic structure.  
\item There is a sequence of characters $\{ \chi_j \}_N^\infty \subset X(\pi_1 M)$ where, for each $j$, $\chi_j$ is a canonical character for $M_j$ and the sequence $\{\chi_j\}$ converges to a canonical character for $M$.
\end{enumerate}
\end{DSthrm}

\begin{prop}\label{regfun} 
Let $L=K \sqcup L'$ be a link in $S^3$, where $K$ is a knot and $L' \neq \emptyset$.  Suppose that $M=S^3-L$ is hyperbolic, $\chi_0\in X(\pi_1 M)$ is a canonical character for $M$, and that $X_0 \subset X(\pi_1 M)$ is the canonical component containing $\chi_0$.  There is a function $C \in \C(X_0)$ and a neighborhood $U_0$ of $\chi_0$ in $X_0$ such that if $\chi\in U_0$ is a canonical character for a hyperbolic orbifold obtained from $M$ by Dehn filling a (possibly empty) collection of cusps corresponding to components of $L'$ then $C$ is smooth at $\chi$ and $C(\chi)\in\mathbb{C}$ represents the cusp parameter for the $K$-cusp of the filled orbifold.
\end{prop}

\proof
Take $\mu, \gamma \in \pi_1 M$ so that $\mu$ is a meridian which encircles the component $K$ of $L$ and  $\gamma$ does not commute with $\mu$.  Define $\mu' = \gamma\mu^{-1}\gamma^{-1}$ and \[U_0 = \{\chi\,|\,I_{[\mu,\mu']}(\chi)\ne 2\}\subset X_0.\]  Suppose that $\rho \in t^{-1}(X_0)$ is a representation such that $\rho(\mu)$ is non-trivial and parabolic.  A brief computation shows that its character is in $U_0$ if and only if $\rho(\mu)$ and $\rho(\mu')$ do not commute.

If $\rho_0 \in R(\pi_1 M)$ has character $\chi_0$ then $\rho_0(\mu)$ is a non-trivial parabolic, and we claim that $\rho_0(\gamma)$ does not share its fixed point at infinity.  If $\rho_0(\gamma)$ is loxodromic, this is a standard consequence of J\o rgensen's inequality. If $\rho_0(\gamma)$ is parabolic, the claim follows from the fact that $\rho_0$ is faithful and parabolic elements of $\mathrm{SL}_2(\mathbb{C})$ commute if and only if they share fixed points at infinity. It follows from the claim that $\rho_0(\mu)$ and $\rho_0(\mu')$ do not share a fixed point in $\bound \H^3$ and hence, do not commute.  Therefore $\chi_0 \in U_0$.

Now, fix a primitive, peripheral element $\lambda \in \pi_1 M$ such that $\langle \lambda, \mu \rangle \cong \Z \times \Z$ and define\begin{align}\label{see?}
	C = \frac{2 I_{\lambda \mu'} - I_\lambda I_{\mu'}}{2 I_{\mu \mu'} - I_\mu I_{\mu'}} \in \C(X_0). \end{align}
Suppose that $M_0$ is a hyperbolic orbifold obtained from $M$ by hyperbolic Dehn filling on a collection of components of $L'$.  We claim that, if $\chi \in U_0$ is a canonical character for $M_0$, then $C(\chi)$ is the cusp parameter of $M_0$.  

Suppose $\chi$ is one such character and $\rho \in t^{-1}(\chi)$.  Since $I_{[\mu,\mu']}(\chi)\ne 2$, the fixed points for $\rho(\mu)$ and $\rho(\mu')$ are distinct.  By composing with an inner automorphism, we may assume that $\rho(\mu)$ fixes $\infty$ and $\rho(\mu')$ fixes $0$.   This forces $\rho(\mu)$ and $\rho(\mu')$ to be upper and lower triangular, respectively.  Finally, we can conjugate $\rho$ by a diagonal matrix in $\text{SL}_2 (\C)$ to assume that the lower left entry of $\rho(\mu')$ is one.  We have now found $\rho \in t^{-1}(\chi)$ such that
\begin{align*}
\rho(\mu) &= \begin{pmatrix} \epsilon & a \\ 0 & \epsilon \end{pmatrix} & \rho(\mu') &= \begin{pmatrix} \eta & 0 \\ 1 & \eta \end{pmatrix} 
\end{align*}
where $\epsilon, \eta \in \{ \pm 1\}$ and $a \neq 0$. Since $\mu$ and $\lambda$ commute, we must also have
\[ \rho(\lambda) = \begin{pmatrix} \kappa & b \\ 0 & \kappa \end{pmatrix} \]
where $\kappa \in \{ \pm 1 \}$.  It follows that the number $b/a$ represents the cusp parameter for the $K$-cusp of $M_0$.

To prove the claim, notice that $C$ takes the finite value $b/a$ when evaluated at $\chi$. Since the denominator $2I_{\mu\mu'} - I_{\mu}I_{\mu'}$ is regular on $X_0$ and non-zero at $\chi$, $C$ is smooth at $\chi$.
\endproof

\begin{cor}\label{TwoCpt}\TwoCptCor\end{cor}

\proof Let $\{M_j\}_1^\infty$ be a set of hyperbolic orbifolds obtained by Dehn fillings on the $K'$-component of $M=S^3-L$ as stated in the assumptions of the corollary.  Let $\{ \chi_j\}_1^\infty$ be a sequence of characters in $X(\pi_1 M)$ as given in (2) of the Hyperbolic Dehn Surgery Theorem above.  In particular, $\chi_j$ is a canonical character for $M_j$ and the sequence converges in $X(\pi_1 M)$ to a canonical character $\chi_\infty$ for $M$.  

By Theorem B.1.2 \cite{BP}, $\chi_\infty$ is a smooth point of $X(\pi_1 M)$ and so lies in a unique algebraic component $X$ of $X(\pi_1)$.  It follows that there is an integer $N$ such that $\{ \chi_j \}_N^\infty \subset X$.

Let $C \in \C(X)$ be the cusp parameter function for the $K$-cusp of $M$ given by Proposition \ref{regfun}.  Since the volume of each $M_j$ is bounded above by the volume of $M$ (see \cite{Th_notes} Theorem 6.5.6.), Lemma \ref{lem: finite image} shows that we may pass to a subsequence of $\{ M_j \}$ to assume that they all have the same cusp shape, which covers a rigid Euclidean orbifold.  Hence the numbers $C(\chi_j) \in \C - \mathbb{R}$ all lie in a single $\text{PSL}_2 (\Z)$ orbit $O$.  Since $C$ is continuous on $X$, $C(\chi_j) \to C(\chi_\infty)$.  Moreover, $O$ is discrete so $C(\chi_j)=C(\chi_\infty)$ for large enough values of $j$.  Since cusp parameters determine the shapes of cusps, the shape of the $K$-cusp of $M$ covers a rigid Euclidean orbifold.

To see that the $K$-cusp of $M$ is geometrically isolated from the $K'$-cusp, we first take $\mu \in \pi_1 M$ to be a meridian encircling $K$.  Recall that $\dim X=2$ and $I_\mu$ is non-constant on $X$.  It follows that there is an algebraic curve $D \subset X$ which contains every $\chi_j$ for large enough $j$.  Since $\chi_j \to \chi_\infty$ in $D$, we conclude that the cusp parameter function $C$ is constant on $D$ which confirms the isolation.  
\endproof

\begin{remark} As noted in the introduction, here ``geometric isolation'' of the $K$-cusp means that at most finitely many fillings of the $K'$-cusp change its shape. The reason for this is that the curve $D$ above may not contain the character of \textit{every} hyperbolic Dehn filling of the $K'$-cusp, though it does contain all but finitely many.\end{remark}

Corollary \ref{TwoCpt} applies if infinitely many knots with hidden symmetries can be obtained by Dehn fillings of $M$ along $K'$, by Proposition 9.1 of \cite{NeumReid}.  It is also worth noting that its first condition holds more generally. 

\begin{cor}\label{NCpt} 
Suppose that $L=K \sqcup L'$ is a hyperbolic link in $S^3$ where $K$ is a knot and $L' $ has $k\geq 1$ components.  If there is a sequence of  hyperbolic Dehn fillings on $L'$ with coefficients in $\mathbb{Z}^{2k}$ approaching $(\infty,\hdots,\infty)$, each of which covers an orbifold with a single rigid cusp, then the shape of the $K$-cusp of $S^3-L$ covers a rigid Euclidean orbifold.
\end{cor}

\proof The proof is essentially that of Corollary \ref{TwoCpt}, less its final paragraph.\endproof

\subsection{From the perspective of the deformation variety}\label{defvar} In this section we will take $M$ to be a complete, oriented hyperbolic $3$-manifold of finite volume equipped with an \textit{ideal triangulation}: a decomposition into ideal tetrahedra, each properly immersed in $M$ so that its interior is embedded and its intersection with any other pulls back to a union of proper faces.  Unfortunately, we do not know the exact class of hyperbolic $3$-manifolds with an ideal triangulation, but it is certainly quite large.  For instance every complete, finite-volume hyperbolic $3$-manifold has a finite-degree cover with this property \cite{LuoSchTill}.

The union of the incomplete and complete hyperbolic structures on $M$ are parametrized by a \emph{deformation variety} $\mathcal{D}_0$. This complex affine algebraic set was first described in Thurston's notes \cite[Ch.~4]{Th_notes}. Neumann--Zagier subsequently used $\mathcal{D}_0$ to obtain an elegant estimate for volume changes of hyperbolic three-manifolds after Dehn surgery \cite{NZ}. 

The definition of $\mathcal{D}_0$ rests on the fact that oriented hyperbolic ideal tetrahedra are parametrized by complex numbers in the upper half-plane up to a small ambiguity. If $\Delta$ is an ideal hyperbolic tetrahedron with a specified edge $e$, it may be isometrically positioned in the  Poincar\'e half-space model for $\H^3$ so that $e$ coincides with the geodesic from $0$ to $\infty$ and the other two ideal vertices coincide with $1$ and a complex number $z$ whose imaginary part is positive. This is because the orientation-preserving isometry group $\mathrm{PSL}(2,\mathbb{C})$ of $\mathbb{H}^3$ acts triply transitively on the extended complex plane $\widehat{\mathbb{C}}$. The number $z$ is called the {\it edge parameter} for $e$. 

The edge parameter is a complete invariant of the isometry class of the pair $(\Delta,e)$. In particular, $z$ does not depend on an orientation of $e$ because the M\"obius transformation $\left(\begin{smallmatrix} 0 & i\sqrt{z} \\ i/\sqrt{z} & 0 \end{smallmatrix}\right)$ which exchanges $0$ and $\infty$ and takes $z$ to $1$ also takes $1$ to $z$. The other edge parameters for $\Delta$ are given by the functions
\begin{align}\label{edge fun}
\zeta_1(z) &= \frac{1}{1-z} & \zeta_2(z) &= \frac{z-1}{z}
\end{align}
as indicated in Figure \ref{fig: shapes}. See eg.~\cite[\S 4.1]{Th_notes}.

\begin{figure}[h] 
   \centering
   \includegraphics[width=1.7in]{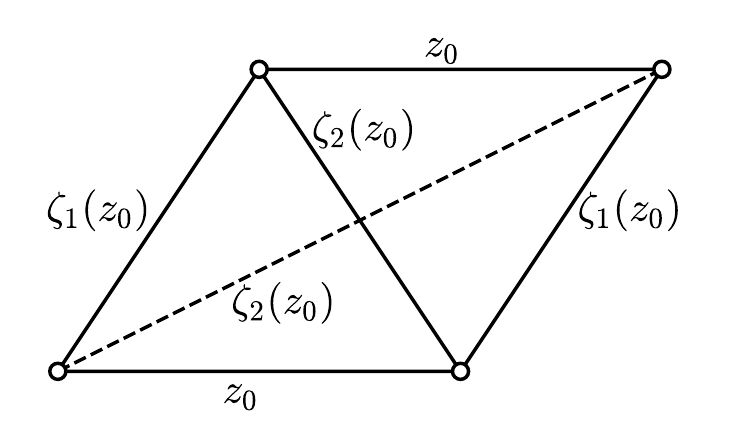} 
   \caption{The shape parameters on a hyperbolic tetrahedron follow a right-hand rule with $\zeta_1$ and $\zeta_2$.}
   \label{fig: shapes}
\end{figure}

Suppose that $\Delta_1,\hdots,\Delta_n$ are the ideal tetrahedra in our ideal triangulation of $M$.  For each $\Delta_i$, fix an edge and assign it an indeterminate parameter $z_i$. For a given (quotient) edge $e$ in the triangulation of $M$, consider the set of all edge parameters of all edges of $\Delta_i$ which coincide with $e$ in $M$.  When we set the product of all such parameters equal to 1 (using every possible $i$), we obtain an equation which is commonly referred to as the \textit{edge equation} for $e$.  If we clear denominators in each edge equation for $M$, we obtain a system of polynomial equations in $n$ variables.  The deformation variety $\mathcal{D}_0 \subset \mathbb{C}^n$ is defined to be the set of solutions to this system of polynomial equations.

For any hyperbolic structure on $M$, the $n$-tuple of edge parameters that records the isometry class of each $\Delta_i$ in $M$ satisfies every edge equation, as these record that the cycle of tetrahedra determined by arranging copies of the $\Delta_i$ around the corresponding edge closes up (see \cite[Fig.~2]{NZ}).  

Most solutions to the edge equations correspond to incomplete hyperbolic structures on $M$.  Indeed, by Mostow--Prasad rigidity, the complete hyperbolic structure on $M$ is unique up to isometry and hence corresponds to a finite, non-empty subset of $\mathcal{D}_0$. (Beware that, if $M$ has automorphisms that do not preserve the triangulation, there may be more than one point of $\mathcal{D}_0$ corresponding to the complete structure.)  Near a point $\mathbf{z}^0\in\mathcal{D}_0$ corresponding to the complete structure, the completeness of the metric at a cusp $c$ of $M$ is recorded by a function $u\co\mathcal{D}_0\to\mathbb{C}^*$ that is analytic around $\mathbf{z}^0$.

The definition of $u$ depends on a function $\mu$ defined in \cite{NZ} in terms of a simplicial path in a cross-section $T$ of $c$.  Here, we will always choose cusp cross-sections so that they intersect any simplex $\sigma$ in the triangulation of $M$ in a collection of triangles each of which bounds a region in $\sigma$ which contains a single ideal vertex.  These cross-sections inherit a triangulation from that of $M$.  

Given a closed oriented simplicial curve $\fg$ in $T$, the function $\mu(\fg)$ on $\mathcal{D}_0$ is defined as follows. Fix a vertex $v_0 \in \fg$ and let $\{ v_0, \ldots, v_{m-1}, v_0\}$ be the ordered sequence of vertices encountered along $\fg$. For each $j$, consider the the tetrahedra which intersect $T$ in triangles that contain a particular vertex $v_j$ and that lie on the righthand side of $\fg$.  Let $w_j \in \C^\ast$ be the product of the edge parameters of the edges of these tetrahedra which pass through $v_j$.  Regarding each $w_j$ as a function of $\bz \in \mathcal{D}_0$, given such $\bz$ we define 
\begin{align}\label{mu fun}
  \mu(\fg)(\bz) = (-1)^m \prod_{j=0}^{m-1} w_{j}. 
\end{align}
See also the definition above Lemma 2.1 in \cite{NZ}.

For any fixed $\fg$, $\mu(\fg)(\cdot) \co \mathcal{D}_0 \to \C$ is clearly a rational function. Lemma 2.1 of \cite{NZ} shows that for any fixed $\bz\in\mathcal{D}_0$, $\mu(\cdot)(\bz) \co \pi_1T_i \to \C^\ast$ is a well-defined homomorphism.  In \cite[\S 4]{NZ}, $\mu(\fg)(\cdot)$ is interpreted as the derivative of the holonomy of $\fg$ under the representation $\pi_1 M\to\mathrm{PSL}(2,\mathbb{C})$ associated to $\bz$. More precisely, any such representation takes any representative of $\pi_1 T_i$ in $\pi_1 M$ to an abelian subgroup of $\mathrm{PSL}(2,\mathbb{C})$ whose elements have a common fixed point $p\in\widehat{\C}$. Upon conjugating so that $p=\infty$, each element $\fg$ acts as $z\mapsto \mu z + \tau$ for some fixed $\mu,\tau\in \mathbb{C}$.  The function $\mu(\fg)$ from (\ref{mu fun}) records this $\mu$.

For each cusp $c_i$ of $M$, fix a simplicial curve $\fg_i$ on a cross-section $T_i$ which represents a element of $\pi_1 T_i$.  Define $u_i(\bz) =\log(\mu(\fg_i)(\bz))$, where the branch of the logarithm is chosen to take the value $0$ at $1$.  See also Equation (28) of \cite{NZ}.  The functions $u_i$ constitute the eponymous {\it good parameters for deformation} of \cite[\S 4]{NZ}. From our perspective, the following is fundamental:

\begin{quote}\it There is a neighborhood $U$ of $\bz^0$ in $\mathcal{D}_0$ such that for all $\bz\in U$, the cusp $c_i$ is complete in the hyperbolic structure determined by $\bz$ if and only if $u_i(\bz) = 0$, i.e. $\mu(\fg_i)(\bz) = 1$.\end{quote}

Note that, since $\mu(\fg_i)$ is a non-constant rational function on $\mathcal{D}_0$, we know that the geometric structures for which $c_i$ is complete lie on a codimension-$1$ subvariety.  

\begin{fact}\label{cross rat} For a Euclidean triangle $\Delta$ in $\mathbb{C}$ with vertices $a$, $b$, $c$ labeled so that $(b-a,c-a)$ is positively oriented as a basis of $\mathbb{R}^2$,
\[ z = \frac{c-a}{b-a} \]
is an orientation-preserving similarity invariant of the pair $(\Delta,a)$. If such a similarity takes $a$ to $0$ and $b$ to $1$ then it takes $c$ to $z$.\end{fact}

Taking $\Delta$ to be the vertical projection to $\mathbb{C}$ of an ideal tetrahedron with vertices at $a$, $b$, $c$ and $\infty$, we recognize $z = (c-a)/(b-a)$ as the parameter of the edge $e$ with endpoints at $a$ and $\infty$.  This fundamental observation is key to our proof of the deformation variety analogue of Proposition \ref{regfun}.

\begin{prop}\label{ratfun}
Suppose that $M$ is a complete oriented hyperbolic 3-manifold with an ideal triangulation decorated by the choice of an edge of each simplex, and $T$ is a cross section of a cusp of $M$. Fix a vertex $v_0$ and edge $f$ of the triangulation of $T$ such that $v_0\in f$.   Given a closed oriented simplicial curve $\fg$ in $T$ containing $v_0$, there is a rational function $\tau(\fg)$ on $\mathcal{D}_0$ which depends on $f$ and has the following properties.
\begin{enumerate}
\item $\tau(\fg)$ is analytic near $\bz^0$.
\item If $\fg$ is not null homotopic then $\tau(\fg)(\bz^0)\neq 0$.
\item If $\fm$ and $\fl$, both containing $v_0$, represent a generating pair for $H_1(T)$ then on a neighborhood $U$ of $\bz^0$,  for all $\bz\in U$ such that $\mu(\fm)(\bz)=1$ the value of $\frac{\tau(\fl)}{\tau(\fm)}$ at $\bz$ represents the cusp parameter of the complete geometrized cusp.
\end{enumerate}
\end{prop}

\begin{dfn}\label{cusp param} For $\fm$ and $\fl$ as in condition (3) above, we call $\tau_c = \frac{\tau(\fl)}{\tau(\fm)}$ a \textit{cusp parameter function} for $c$. \end{dfn} 

\begin{proof}  
The definition of $\tau$, given in (\ref{tau}) below, is very similar to that of $\mu$.

For a closed oriented simplicial curve $\fg$ in $T$ containing $v_0$, let $\{ v_0, \ldots, v_{m-1}, v_0\}$ be the ordered sequence of vertices encountered along $\fg$. Consider the tetrahedra which intersect $T$ in triangles that contain a vertex of $\fg$ and lie on the righthand side of $\fg$.  For $j \neq 0$, let $w_j \in \C^\ast$ be the product of the edge parameters of the edges of these tetrahedra which pass through $v_j$.  As indicated in Figure \ref{triang path}, define $w_0$ to be the product of parameters of edges through $v_0$ which correspond to tetrahedra which intersect $T$ in a triangle that lies between $f$ and the first edge of $\fg$.   If $f$ is the first edge of $\fg$, take $w_0=1$.    Given $\bz \in \mathcal{D}_0$, define
\begin{align}\label{tau} \tau(\fg)(\bz) = \sum_{l=0}^{m-1} (-1)^l \, \prod_{j=0}^l w_j.\end{align}

The function $\tau(\fg)$, which is evidently rational, is analytic near $\bz^0$.  This is because the coordinates of $\bz^0$ are bounded away from $0$ and $1$, so the functions $\zeta_1$ and $\zeta_2$ are analytic on the coordinates of $\bz$ near $\bz^0$. 

To understand the geometric significance of $\tau$, first recall that each $\bz\in\mathcal{D}_0$ determines a hyperbolic structure on $M$ and a developing map $D_{\bz}\co\widetilde{M}\to\mathbb{H}^3$, where $\widetilde{M}$ is the universal cover of $M$.  The tetrahedra immersed in $M$ lift to embeddings in $\widetilde{M}$, and these develop to ideal tetrahedra in $\mathbb{H}^3$.  Associated to $D_\bz$ is a holonomy representation $\rho_{\bz}\co\pi_1 M\to\mathrm{PSL}(2,\mathbb{C})$ which satisfies the equivariance condition $D_{\bz}(\gamma.x) = \rho_{\bz}(\gamma).D_{\bz}(x)$ for every $\gamma \in \pi_1 M$ and every $x \in \widetilde{M}$.

Fix a component $\widetilde{T}$ of the preimage of $T$ in $\widetilde{M}$ and let $\Gamma$ be the stabilizer of $\widetilde{T}$ in $\pi_1 M$.  Fix a point $\tilde{v}_0$ in the preimage of $v_0$ and a lift $\tilde{f}$ of $f$ with an endpoint at $\tilde{v}_0$.  Every edge in $M$ which meets $T$ lifts to a collection of edges in $\widetilde{M}$ which meet $\widetilde{T}$.  Their developed images all share a common ideal vertex $p_\infty \in \widehat{\C}$ which corresponds to the ends of the edges in $M$ which exit the cusp.  After post-composing $D_{\bz}$ with an isometry and conjugating $\rho_{\bz}$ by the same isometry, we may assume that $p_{\infty}=\infty$, that the tetrahedral edge containing $D_{\bz}(\tilde{v}_0)$ has its other ideal vertex at $0$, and that the tetrahedral face containing $D_{\bz}(\tilde{f})$ has its final ideal vertex at $1$.

\begin{figure}[h] 
   \centering
   \includegraphics[width=2.6in]{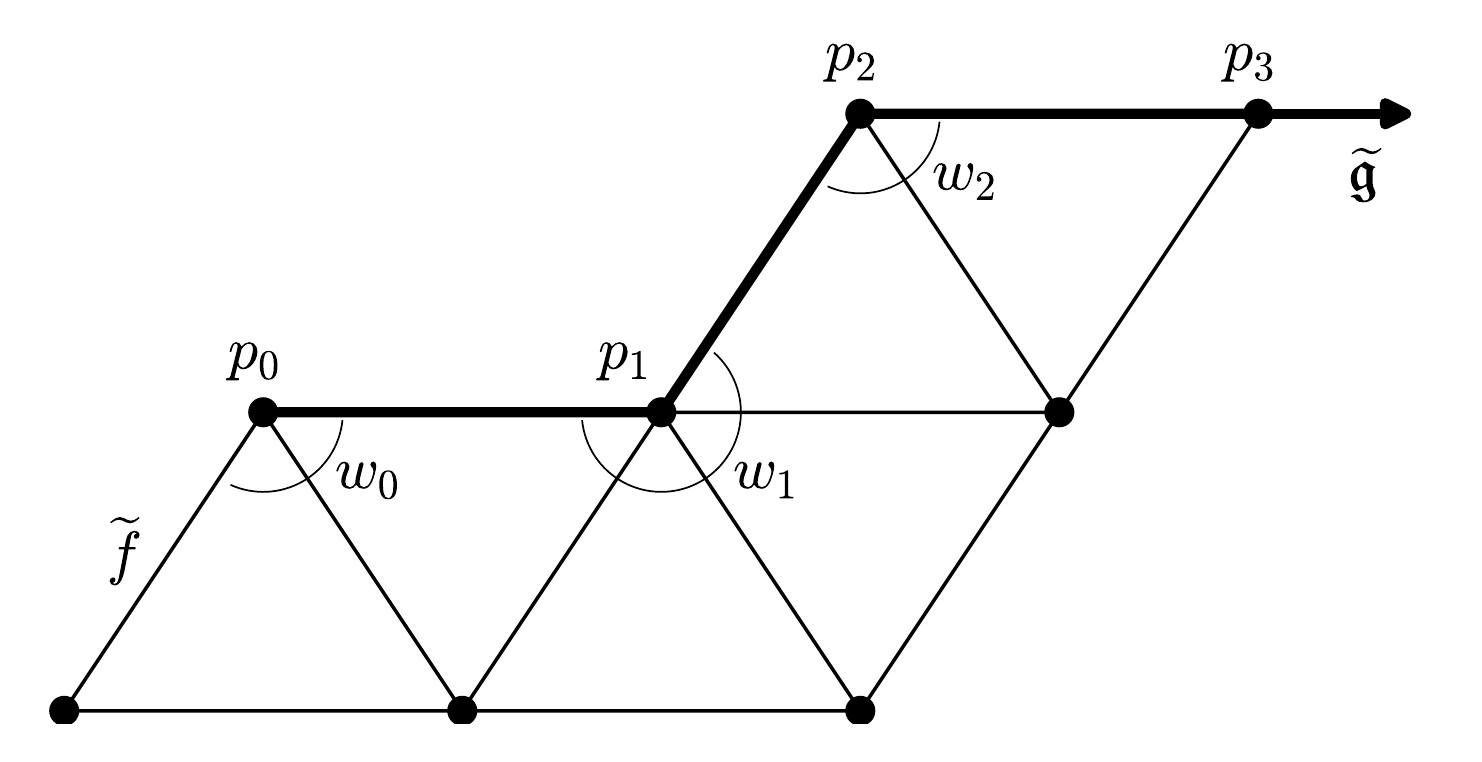} 
   \caption{Part of a lifted, developed path in a cusp cross-section.}
   \label{triang path}
\end{figure}

%
%
%
%
%

Take $\fg$ and $\{ v_0, \ldots, v_{m-1}, v_0\}$ as above and let $\tilde{\fg}$ be the lift to $\widetilde{T}$ of $\fg$ starting at $\tilde{v}_0$.  Label the vertices of $\tilde{\fg}$ by $\tilde{v}_j$ in order as we encounter them.  Define $p_j$ to be the ideal endpoint in $\C$ of the developed image of the tetrahedral edge in $\widetilde{M}$ which contains $\tilde{v}_j$.  For each $j$ strictly between $0$ and $m$, Fact \ref{cross rat} quickly implies that
\[	\frac{p_{j+1}-p_j}{p_{j-1}-p_j} =  w_j. \]
When $j=0$, the corresponding formula reduces to $p_1=w_0$.  This yields the formula
\[ p_{l+1}-p_l = (-1)^l \, \prod_{j=0}^l w_j\]
which is valid for every $l \in \{ 0, \ldots ,  m-1\}$.  Summing from $0$ to $m-1$, we find that $\tau(\fg)(\bz) = p_m$. 

We claim that $\tau(\fg)(\bz)$ is the translational part of $\rho_\bz(\gamma)$, where $\gamma$ is the element of the stabilizer $\Gamma$ of $\widetilde{T}$ (a copy of $\pi_1 T$ acting on $\widetilde{M}$ by covering transformations) represented by the curve $\fg$. Here we note that the elements of $\rho_{\bz}(\Gamma)$ all have $p_{\infty} = \infty$ as a fixed point in $\widehat{\mathbb{C}}$, so each is an affine transformation of $\C$. In particular, $\rho_{\bz}(\gamma)$ has the form $z \mapsto \mu z + \tau$ for some $\mu\neq 0$ and $\tau\in\mathbb{C}$, its \textit{translational part}, which can thus be recovered by evaluating $\rho_{\bz}(\gamma)$ at $z=0$. On the other hand, since $\tilde{\fg}$ is a lift of $\fg$ we have $\rho_\bz(\gamma).D_\bz(\tilde{v}_0)=D_\bz(\tilde{v}_m)$, and by our choice of normalization this implies $\rho_\bz(\gamma)(0)= p_m = \tau(\fg)(\bz)$, proving the claim.

The holonomy representation $\rho_{\bz^0}$ of the complete structure is faithful and takes non-trivial elements of $\Gamma$ to parabolic elements, so since $\rho_{\bz^0}(\gamma)$ is the function $z\mapsto z +  \tau(\fg)(\bz^0)$ and is not the identity, $\tau(\fg)(\bz^0)\neq 0$. This proves assertion (2).

%

Now let $\fl$ and $\fm$ be simplicial curves in $T$, both containing $v_0$, which represent a generating pair for the first homology of $c$.  Because $\tau(\fm)(\bz^0)\neq0$, the function $\tau(\fm)$ is non-zero on a neighborhood $U$ of $\bz^0$.  Let $K = U \cap \{ \bz \in \mathcal{D}_0 \, | \, \mu(\fm)=1 \}$.  As recorded above Fact \ref{cross rat}, $c$ is complete under the geometry determined by $\bz \in U$ if and only if $\bz \in K$; in particular $\bz^0 \in K$.  Define
\[ \tau_c = \frac{\tau(\fl)}{\tau(\fm)}\]
and note that $\tau_c$ does not have a pole on $U$.

For any $\bz \in K$ and $\gamma \in \Gamma$, $\rho_{\bz}(\gamma)$ is an affine transformation of the form $z \mapsto \alpha z + \beta$, for some $\alpha\ne 0, \beta \in \C$.  On the other hand, $\rho_\bz$ takes the element $\fm \in \Gamma$ represented by $\fm$ to the transformation $z \mapsto z + \tau(\fm)(\bz)$, since $\mu(\fm)(\bz)=1$.  Hence, the translational parts of $\rho_\bz(\fm \gamma)$ and $\rho_\bz(\gamma \fm)$ are 
\[\beta+\tau(\fm)(\bz) \qquad \text{and} \qquad \alpha \tau(\fm)(\bz)+\beta.\]
However $\Gamma$ is abelian, so these quantities must be the same and $\alpha=1$.  Now if we take $\gamma$ to be the element of $\Gamma$ represented by $\fl$, we see that $\mu(\fl)=1$ on $K$.  This shows that $\rho_{\bz}(\Gamma)$ is a lattice generated by the translations $z\mapsto z + \tau(\fm)(\bz)$ and $z\mapsto z + \tau(\fl)(\bz)$.  Therefore, by definition of the complex modulus, $\tau_c(\bz)$ represents the complex modulus of $c$.\end{proof}


\begin{remark} Unless $\rho_{\bz}$ and $D_{\bz}$ are normalized as in the proof of Proposition \ref{ratfun}, $\tau(\gamma)(\bz)$ is not the translational part of $\rho_{\bz}(\gamma)$. Indeed, if we still take $p_{\infty} = \infty$ but let $p_0\in\mathbb{C}$ be the other ideal endpoint of the tetrahedral edge through $D_{\bz}(\tilde{v}_0)$ and $q\in\mathbb{C}$ the final ideal vertex of the face containing $D_{\bz}(\tilde{f})$ then the proof above gives
\[ \tau(\gamma)(\bz) = \frac{p_m-p_0}{q-p_0} = \frac{(\mu-1)p_0+\tau}{q-p_0}, \]
where $\rho_{\bz}(\gamma)$ is the map $z\mapsto \mu z+\tau$. In particular, if $\mu = 1$ then $\tau(\gamma)(\bz)$ is the translational part, normalized by the quantity $q-p_0$ determined by $D_{\bz}(\tilde{f})$.\end{remark}

Here is an analog of Corollary \ref{TwoCpt} in the context of the deformation variety.

\begin{cor}\label{TwoCpt def} Let $M$ be a complete, oriented hyperbolic $3$-manifold with two cusps $c$ and $c'$, and an ideal triangulation decorated by the choice of an edge of each simplex.   Fix a cross section $T$ of the cusp $c$ and simplicial curves $\mathfrak{m}$ and $\mathfrak{l}$ representing a generating pair for $H_1(T)$, and let $\tau_c = \tau(\fl)/\tau(\fm)$ be a cusp parameter function for $c$.  If infinitely many one-cusped orbifolds produced by hyperbolic Dehn filling on $c'$ cover orbifolds with rigid cusps then for $\bz^0 \in \mathcal{D}_0$ corresponding to the complete hyperbolic structure on $M$:
\begin{enumerate}
	\item $\tau_c(\bz^{0}) \in\mathbb{Q}(i)$ or $\tau_c(\bz^{0}) \in \mathbb{Q}(\sqrt{-3})$ and
	\item $\tau_c$ is constant on the irreducible component of $\{\bz\in\mathcal{D}_0(M)\,|\,\mu(\mathbf{\fm})(\bz)=1\}$ containing $\bz^0$.
	\end{enumerate}
In particular, this holds if infinitely many hyperbolic knot complements in $S^3$ with hidden symmetries can be produced from $M$ by Dehn filling $c'$.\end{cor}

\begin{proof} The proof tracks that of Corollary \ref{TwoCpt}. In the same way, it follows from Lemma \ref{lem: finite image} that if $M$ has infinitely many fillings of $c'$ covering orbifolds with rigid cusps then $\tau_c$ is constant on a set $\{\bz^i\}\in\mathcal{D}_0(M)$ that accumulates at $\bz^0$. For each such $i$, $\tau_c(\bz^i)$ is the complex modulus of a Euclidean torus that covers a $(2,4,4)$-, $(2,3,6)$-, or $(3,3,3)$-triangle orbifold, so either $\tau_c(\bz^i)\in\mathbb{Q}(i)$ (in the first case) or $\tau_c(\bz^i)\in\mathbb{Q}(\sqrt{-3})$ (in the latter two). Taking a limit as $i\to\infty$ yields criterion (1).

Each filling has a complete cusp at $c$, so the points $\bz^i$ all lie in the algebraic subset $V$ defined by $\mu(\mathbf{\fm})=1$.   By Section 4 of \cite{NZ}, $\log \mu(\mathbf{\fm})$ is analytic on a neighborhood $U$ of $\bz^0$ in $\mathcal{D}_0(M)$.  On $U$, the locus  $\{\log \mu(\mathbf{\fm})=0\}$ is identical to $\{ \mu(\mathbf{\fm})=1 \}$ and the former is a smooth codimension-one analytic submanifold of $U$.  This means that there is a unique algebraic component $V_0$ of $V$ which contains $\bz^0$.  Furthermore $V_0$ is smooth at $\bz^0$.  Since $\bz^i \to \bz^0$, the tail of this sequence must be contained in $V_0$.  The function $\tau_c$ is rational on the curve $V_0$ and constant on the tail of the sequence $\{ \bz^i \}$.  Therefore, $\tau_c$ must be constant on $V_0$.
\end{proof}

\section{An example from the literature: the Whitehead link}\label{whitehead} 

In this mostly expository section, we consider the complement $M_W$ of the  link $5^2_1$. This is the two-component hyperbolic link with fewest crossings and is commonly known as the Whitehead link. $M_W$ is arithmetic and covers the Bianchi orbifold $\mathbb{H}^3/\mathrm{PSL}(2,\mathcal{O}_1)$ so its cusps each satisfy condition (1) of Corollaries \ref{TwoCpt} and \ref{TwoCpt def}. We will show that they fail the second conditions.  This implies the following fact. 

\begin{fact}\label{whitehead surg} At most finitely many knot complements with hidden symmetries arise as hyperbolic Dehn fillings on the Whitehead link.\end{fact}

Actually, more is known in this case.  The knot complements produced by surgery on one component of $M_W$ are complements of twist knots, and by \cite{ReidWalsh}, the only one of these which has hidden symmetries is the complement of the figure-eight knot.  The Whitehead link is well-studied, and we take this opportunity to draw together some threads from the literature and compare the perspectives of both Sections \ref{charvar} and \ref{defvar} in the context of this familiar example.  We begin with a general observation.

\begin{fact}\label{involution} For complete hyperbolic $3$-manifold $M$ with two cusps, $c_1$ and $c_2$, and an involution that exchanges them, $c_1$ is geometrically isolated from $c_2$ if and only if $c_2$ is geometrically isolated from $c_1$.\end{fact}

Theorem 3 of \cite{NR} shows that in general, geometric isolation is not a symmetric relation. But, by Mostow rigidity, an automorphism of $M$ taking $c_1$ to $c_2$ induces a bijective, isometric correspondence from the one-cusped hyperbolic manifolds obtained by filling $c_1$ to those obtained by filling $c_2$, which establishes the fact.  It is well known that $M_W$ has an involution that exchanges its two cusps, so, in what follows, we focus on a fixed cusp of $M_W$.
 
\begin{figure}[h] 
   \centering
   \includegraphics[width=4.3in]{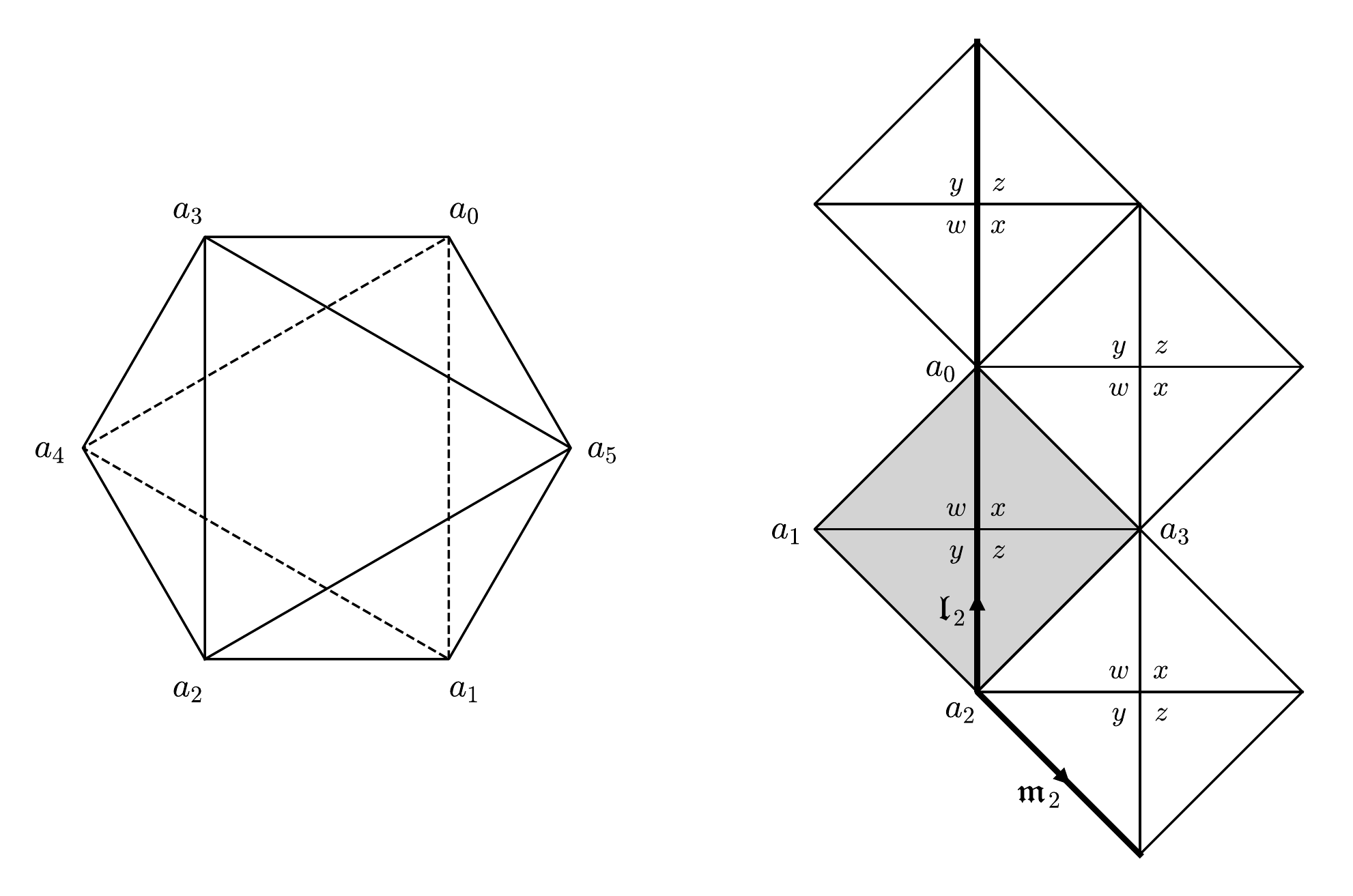} 
   \caption{We label the vertices of $\mathcal{O}$ as indicated on the left.  On the right, $
   \mathcal{O}$ (shaded) is placed in $\mathbb{H}^3$ so that $(a_3,a_4,a_5) = (1, \infty, 0)$.} 
   \label{fig: oct}
\end{figure}

Let $\mathcal{O}$ be a regular ideal octahedron embedded in Euclidean 3-space with ideal vertices labeled as in Figure \ref{fig: oct}.  Let $\phi, \psi, \eta$, and $\rho$ be the orientation preserving Euclidean isometries with the following actions on ordered triples of vertices of $\mathcal{O}$.
\begin{align*}
\phi(a_0,a_3,a_5)&=(a_1,a_0,a_4) & \psi(a_2,a_3,a_5)&=(a_1,a_0,a_5) \\
\eta(a_2,a_3,a_4)&=(a_1,a_2,a_5) & \rho(a_0,a_3,a_4)&=(a_1,a_2,a_4)
\end{align*}
As shown in \cite{NeumReid}, $M_W$ is homeomorphic to the identification space given by the face pairings $\{ \phi, \psi, \eta, \rho \}$.

\begin{remark}
Although we are following \cite{NeumReid}, our figure is the mirror image of their Figure 17.  We define $\mathcal{O}$ in this way to account for the fact that they use a left-handed convention for $\zeta_1$ and $\zeta_2$.
\end{remark}

\subsection{The deformation variety}\label{whitehead def} Here, we use the strategy of Proposition \ref{ratfun} to perform the elementary calculation alluded to in the proof of Theorem 6.1 in \cite{NeumReid}.  This calculation produces the formula below, valid on the the curve $V_0$ of points in $\mathcal{D}_0(M_W)$ where $c_1$ is complete, for a cusp parameter function for a cusp $c_1$ of $M_W$.  
\begin{align}\label{NR param} \tau_{c_1} = \frac{4x}{1-x^2} - 2 \end{align}
By \cite{NeumReid}, the single parameter $x$ above parametrizes $V_0$. Thus $\tau_{c_1}$ is evidently non-constant on $V_0$, proving that $c_1$ is not geometrically isolated from the other cusp $c_2$ of $M_W$. 

Our first step is to add a central edge from $a_4$ to $a_5$ in $\mathcal{O}$, thus decomposing $\mathcal{O}$ into a union of four ideal tetrahedra.  Following \cite{NeumReid}, we assign indeterminant edge parameters $w,x,y$, and $z$ to the edges $[a_0,a_1]$, $[a_0,a_3]$, $[a_1,a_2]$, and $[a_2,a_3]$ respectively.  Then, if $(w,x,y,z) \in \C^4$ is a point in $\mathcal{D}_0(M_W)$, we can place $\mathcal{O}$ in $\mathbb{H}^3$ so that 
\[(a_0,a_1,a_2,a_3,a_4,a_5) = (x,wx,wxy,1,\infty,0).\]
When we develop the octahedron across $\mathbb{H}^3$ using the isometries induced by the face pairings $\phi, \psi, \eta$, and $\rho$, we construct the diagram on the right in Figure \ref{fig: oct}.  It shows the lift to $\C$ of the induced triangulation of a cross-section of  $c_2$.  The parameters $w,x,y$, and $z$ label the edges of the tetrahedra to which they are assigned in the Neumann-Reid triangulation.  Since we use the mirror image of the octahedron in \cite{NeumReid} and Neumann and Reid view the cusp from outside the manifold, the right-hand part of our Figure \ref{fig: oct} is combinatorially identical to the right-hand part of their Figure 14.  

 A 4-tuple $(w,x,y,z) \in \C^4$ is in $\mathcal{D}_0(M_W)$ if and only if the gluing equations 
\[1=wxyz \quad \text{and} \quad 1= \zeta_1(w) \zeta_1(x) \zeta_1(y) \zeta_1(z) \zeta_2(w)^2 \zeta_2(x)^2
\]
are satisfied.  This is equivalent to requiring $z=(xyz)^{-1}$ and  
\[ 0= w^2x^2y^2-w^2x^2y+wx-wy-xy+y.\]
As in \cite{NeumReid}, the heavy oriented simplicial paths labeled $\mathfrak{l}_2$ and $\mathfrak{m}_2$ represent generators for the first homology of the cusp $c_2$.

By applying the aforementioned cusp swapping involution, we obtain a corresponding diagram for the cusp $c_1$.  This diagram differs from that in Figure \ref{fig: oct} only by interchanging the edge labels $y$ and $z$ (see also the left image of Figure 14 in \cite{NeumReid}).   In the diagram for $c_1$, we refer to the simplicial curves which correspond to $\mathfrak{l}_2$ and $\mathfrak{m}_2$ as $\mathfrak{l}_1$ and $\mathfrak{m}_1$.  As with $c_2$, these curves form a basis for the first homology of the cusp $c_1$.

We know that the equation $\mu(\mathfrak{m}_1) =1$ determines the subset $V_0$ of $\mathcal{D}_0(M_W)$ for which the cusp $c_1$ is complete.  Neumann and Reid show that $V_0$ is parametrized by $x$ using the equality
\[ (w,x,y,z) = \left( -1/x, x, -1/x,x\right).\]
Equation \ref{NR param} gives a a cusp parameter function on $V_0$ which records the complex modulus of the shape of $c_1$ relative to $\fm_1$ and $\fl_1$.

Our computation begins by selecting the single edge $\fm_1$ as the reference edge $f$ of Proposition \ref{ratfun}, and thus obtaining $\tau(\fm_1) = 1$ without computation.  The lifted representative of $\fl_1$ has five vertices, which we call $\tilde{v}_0,\hdots,\tilde{v}_4$ indexed in agreement with the orientation of $\fl_1$.  The vertices $\tilde{v}_0$ and $\tilde{v}_4$ lie in the preimage of a single vertex in the cusp cross-section.  

To compute the values $\{ w_j \}_0^3$ needed in our formula for $\tau(-\fl)$, we use the functions $\zeta_1$ and $\zeta_2$ from (\ref{edge fun}), the usual right hand convention shown in Figure \ref{fig: shapes}, and a diagram of a cross-section of $c_1$ obtained from Figure \ref{fig: oct} by swapping labels $y$ and $z$ and substituting $- \frac1x$ for $w$ and $y$.  This yields
\begin{align*}
w_0&=\zeta_1(x) \zeta_2\left(-1/x\right) \zeta_1\left(-1/x\right)=\frac{x}{1-x} & w_1&=x \left( -1/x \right) = -1 \\
w_2&= \zeta_2(x) \zeta_1\left(-1/x\right) \zeta_2(x) \zeta_1(x) = \frac{1-x}{x(x+1)} & w_3&=x\left(-1/x\right)=-1
\end{align*}
and 
\begin{align*} 
\tau_{c_1} &= \tau(\fl_1) = w_0-w_0w_1+w_0w_1w_2-w_0w_1w_2w_3 \\ &= \frac{4x}{1-x^2}-2,\end{align*}
recovering equation (\ref{NR param}) as desired.

\subsection{The character variety}\label{whitehead char} Next, using the techniques of Section \ref{charvar}, we compute a cusp parameter function on the character variety of $M_W$.  Recall that $M_W$ is the identification space of the octahedron $\mathcal{O}$ by the face pairings $\phi, \psi, \eta$, and $\rho$.  Using Poincar\'e's polyhedron theorem, we obtain the presentation \[\left\langle \phi, \psi, \eta, \rho \, \big| \, \phi\psi=\rho \phi, \, \rho \phi^{-1} \psi = \eta, \, \psi \eta = \eta \rho \right \rangle\] for $\pi_1 M_W$.  If we let $a=\phi \psi^{-1}$, $b=\psi$, and $\omega = ba(bab)^{-1}ab$, this simplifies to the presentation $\pi_1 M_W =\langle a,b \, \big| \, a\omega = \omega a \rangle$, which coincides with that given in \cite[\S 5]{HLM}.  By considering the action of the face pairings on $\mathcal{O}$, we find that every neighborhood of $a_4$ and $a_5$ enters the cusp $c_2$. All other vertex neighborhoods enter the cusp $c_1$.  Since $a$ fixes $a_0$ and $b$ fixes $a_5$, $a$ and $b$ are meridians for $c_1$ and $c_2$ respectively.

Let \[ p= I_aI_b - (I_a^2+I_b^2-2)I_{ab} + I_aI_bI_{ab} ^2 - I_{ab} ^3. \]
As shown in \cite{HLM}, the canonical component $X_0$ of $X(\pi_1 M_W)$ is algebraic set in $\C^3$ determined by $p$.  

Every hyperbolic Dehn filling of $c_2$ under which $c_1$ remains complete has a $\text{SL}_2 \C$  holonomy character on $X_0$ which satisfies $I_a=2$.  The corresponding specialization of $p$ factors as
\[  (I_{ab} ^2-I_{ab} I_b+2)(I_b-I_{ab} ).\]
Characters which satisfy $0=I_b-I_{ab}$ are characters of reducible representations (cf. \cite[\S 5]{HLM}), so the holonomy character for each hyperbolic Dehn filling of $c_2$ satisfies
\[  0=I_{ab} ^2-I_{ab} I_b+2.\]
We refer to the curve of solutions to this equation as $W_0$.  Because this equation expresses $I_b$ as a rational function in $I_{ab}$, $I_{ab}$ parametrizes $W_0$.

Take $\mu = a$, $\mu' = b^{-1}a^{-1} b$, and $\lambda = w$.  Then $I_\mu=I_{\mu'}=I_a$, $\lambda \mu'=[b,a]$, and $\mu\mu'=\left[a,b^{-1}\right]$.  Using trace relations (see eg.~\cite[\S 3.4]{MacMat}), we obtain
\[ I_{\lambda\mu'} = I_{\mu\mu'} = I_a^2+I_b^2+I_{ab}^2-I_aI_bI_{ab}-2. \] 
On $W_0$,  this can be expressed as a function of $I_{ab}$, namely
\[ I_{\lambda\mu'} = I_{\mu\mu'} = 2 I_{ab}^{-2} \, (I_{ab}^2+2).\]
More trace relations (or matrix computations) show that $I_\lambda=-2$ on $W_0$.  Using these values in Equation (\ref{see?}) from Proposition \ref{regfun} yields a cusp parameter function \begin{align}\label{see} C = I_{ab}^2+1.\end{align}  Since $W_0$ is parametrized by $I_{ab}$, $C$ is non-constant.

\subsection{Reconciling the cusp parameter functions}
Next we verify that the formulas (\ref{NR param}) and (\ref{see}) differ by a M\"obius transformation, after expressing the character $I_{ab}$ in terms of the tetrahedral parameter $x$. (This is tantamount to composing $C$ with a map $\Theta$ from the curve $V_0 \subset \mathcal{D}(M_W)$ to $W_0\subset X(\pi_1 M_W)$, as described by eg.~Champanerkar \cite{Champ}.) Suppose that $(w,x,y) \in \mathcal{D}_0(M_W)$ and recall that
\begin{align*}
 \phi(a_0, a_3, a_5) &= (a_1, a_0, a_4), & \psi(a_2, a_3, a_5)&= (a_1, a_0, a_5),
 \end{align*}
$a=\phi \psi^{-1}$, and $b=\psi$.  We have placed $\mathcal{O}$ in $\mathbb{H}^3$ so that 
\[ (a_0, a_1, a_2, a_3, a_4, a_5) = (x,wx, wxy, 1, \infty, 0). \]
The face pairings $\phi$ and $\psi=b$ can now be represented as hyperbolic isometries
\[ \phi= \begin{bmatrix} x(wx-1) & -x^2(w-1) \\ x-1 & 0\end{bmatrix} \quad \text{and} \quad b= \begin{bmatrix} x(wxy-1) & 0 \\ xy-1 & xy(w-1)\end{bmatrix},
\]
viewed as matrices in $\text{PGL}_2 \C$.  Likewise $a$ and $ab$ are represented by
\[\begin{bmatrix} x(wxy+xy-y-1) & -x^2(wxy-1) \\ y(x-1) & 0\end{bmatrix}  \quad \text{and} \quad \begin{bmatrix} x(wx-1) & -x^2(w-1) \\ x-1 & 0 \end{bmatrix}.
 \]
To restrict our attention to $V_0$, we set $w=y=-1/x$ and obtain
\begin{align*}
a&= \begin{bmatrix} 2x& -x^2 \\ 1 & 0 \end{bmatrix} & b&= \begin{bmatrix} x(x-1) & 0 \\ 2x & -x-1 \end{bmatrix} &
ab&= \begin{bmatrix} -2x & x(x+1) \\ x-1 & 0 \end{bmatrix}.
\end{align*}
 When we normalize these matrices to have determinant one and take their traces, we get a map $\Theta \co V_0 \to X(\pi_1 M_W)$ given by
 \begin{align*}
 I_a&= 2 &
 I_b&=\frac{x^2-2x-1}{\sqrt{x(1-x^2)}} &
 I_{ab}&=\frac{-2x}{\sqrt{x(1-x^2)}}.
 \end{align*}
 It is easy to see that the image of $\Theta$ is contained in $W_0$.  Moreover,
\[C \Theta(x) = \frac{4x}{1-x^2} = \tau_{c_1} + 2\]
as desired. (Recall that changing generators for $H_1(\mathrm{cusp})$ changes the cusp parameter by an element of $\mathrm{PSL_2}(\mathbb{Z})$.)

\section{The link $6^2_2$}\label{sixtwotwo}
In this section, we study the link pictured in Figure \ref{$6^2_2$}.  This link is commonly known as the link $6_2^2$ from Rolfsen's table \cite[App.~C]{Rolfsen}) and also as the two-bridge link $L\left( \frac{3}{10} \right)$.  Let $M$ denote its complement in $S^3$.  

This link is exceptional in that it is a two-bridge link with very low crossing number and its complement is arithmetic and commensurable with the Bianchi orbifold $\H^3/\mathrm{PSL}(2,\mathcal{O}_3)$, where $\mathcal{O}_3$ denotes the ring of integers for $\mathbb{Q}(\sqrt{-3})$.  It is well known that there is a $\mathrm{PSL}(2,\mathcal{O}_3)$-invariant tiling of $\H^3$ by regular ideal tetrahedra. It follows that, like those of the Whitehead link complement, the cusp cross-sections of $M$ cover rigid Euclidean orbifolds.

Next, we study the deformation variety $\mathcal{D}_0(M)$.    As in the previous section, we intend to find a formula for a cusp parameter function for a cusp of $M$ and to show that this function is non-constant on the locus in $\mathcal{D}_0(M)$ where the other cusp remains complete.  As before, this will show that at most finitely many knot complements with hidden symmetries arise as hyperbolic Dehn surgeries on the link $6^2_2$.

We first triangulate $M$ using a procedure for triangulating two-bridge link complements which was laid out by Sakuma and Weeks in Chapter II of \cite{SW}. 

\begin{figure}[h]
\begin{tikzpicture}[scale=.27]

\begin{scope}[xshift=-3in]
\draw[ultra thick](.30,-.55) to [out=135, in=-90] (0,0)  to [out=90,in=-135] (1.65,1.45) to [out=45,in=-45] (2.15,4);
\draw [ultra thick](.55,-.95) to [out=-45,in=-30] (-2,-3) to [out=160, in=-90](-4,0) to [out=90, in=-230] (-1.75,4.25);
\draw[ultra thick](-.17,4.67) to [out=-230,in=150](-2,1.5) to [out=-40, in=130] (.15,.75);
\draw[ultra thick](.42,.54) to [out=-50, in=50] (.13,-1.25) to [out=-130,in=-40] (.30,-1.75);
\draw[ultra thick](-1.50,3.88) to [out=-50,in=-155](.55,4.8) to [out=25,in=135](1.85,4.4);
\draw[ultra thick](0.21,4.23) to [out=-45,in=-180](1.3,3.7) to [out=0,in=-135](2.26,4.5)to [out=45,in=110](4.75,1.5) to [out=-70,in=50](4.20,-2) to [out=-130,in=-45](.70,-2.10);
\end{scope}

\end{tikzpicture}
\caption{The link $6^2_2$.}
\label{$6^2_2$}
\end{figure}
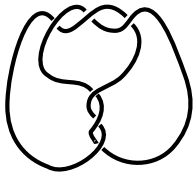

For $i\in\{1,2,3,4\}$, let $X_i$ be a copy of $\mathbb{R}^2 - \mathbb{Z}^2$ each decorated with the set of all lines through $\Z^2$ which have slopes belonging to a set of three given slopes.  For each $i$, the slopes are as follows
\begin{align*}
	& X_1: \{0,1,1/2\} && X_2 : \{0,1/2,1/3\} && X_3 : \{0,1/3,1/4\} && X_4 : \{1/3,1/4,2/7\}. \end{align*}
	
The union of these lines divides $X_i$ into ideal triangles. To see this, notice that the three slopes are vertices of a Farey triangle, so there is an element of $\mathrm{SL}_2(\mathbb{Z})$ taking the lines on $X_i$ to the {\it grid lines} on $\mathbb{R}^2$ with slopes $0$, $1$, and $\infty$.  It is obvious that these grid lines triangulate the plane.

For each $i<4$, we identify $X_i$ to $X_{i+1}$ along the lines they share. For any fixed $i$, the line segments in $X_i$ but not $X_{i+1}$ are called {\it top edges}, those in $X_{i+1}$ but not $X_i$ are called {\it bottom edges}, and the edges common to both $X_i$ and $X_{i+1}$ are called {\it side edges}. Again for fixed $i$, cutting the identification space of $X_i$ and $X_{i+1}$ along its side edges yields a collection of copies of the boundary of an ideal tetrahedron.  We fill each such copy with an ideal tetrahedron in the identification space of all the $X_i$, calling the resulting complex of ideal tetrahedra $\widetilde{W}$.

$\widetilde{W}$ is homeomorphic to $(\mathbb{R}^2-\mathbb{Z}^2)\times [0,1]$, with three layers of ideal tetrahedra, one between $X_i$ and $X_{i+1}$ for each $i<4$. Let $P$ be the \textit{pillowcase group} generated by Euclidean rotations of angle $\pi$ around the points of $\mathbb{Z}^2$.  The action of $P$ on $\mathbb{R}^2$ preserves the integer lattice and the slopes of lines, hence it preserves $X_i$ and its decoration by lines.  Thus it extends simplicially to $\widetilde{W}$.  Let $W = \widetilde{W}/P$ be the corresponding quotient.

The action of $P$ on $\widetilde{W}$ is free and properly discontinuous, so $W$ inherits the structure of a manifold with boundary.  In particular, if we let $\Sigma_i$ be the four-punctured sphere $X_i/P$, then $\partial W = \Sigma_1\sqcup \Sigma_4$.   Since the action of $P$ on $\mathbb{R}^2$ preserves slopes of lines, each edge in the ideal triangulation of $\Sigma_i$ has a well-defined slope. Appealing to Theorem II.2.4 of \cite{SW}, we observe that the link complement $M$ is homeomorphic to the quotient of $W$ given by folding $\Sigma_1$ along edges of slope $1/2$ and $\Sigma_4$ along edges of slope $1/4$, identifying the pairs of triangles sharing the edges.

 \begin{figure}[h] 
   \centering
   \includegraphics[width=2.7in]{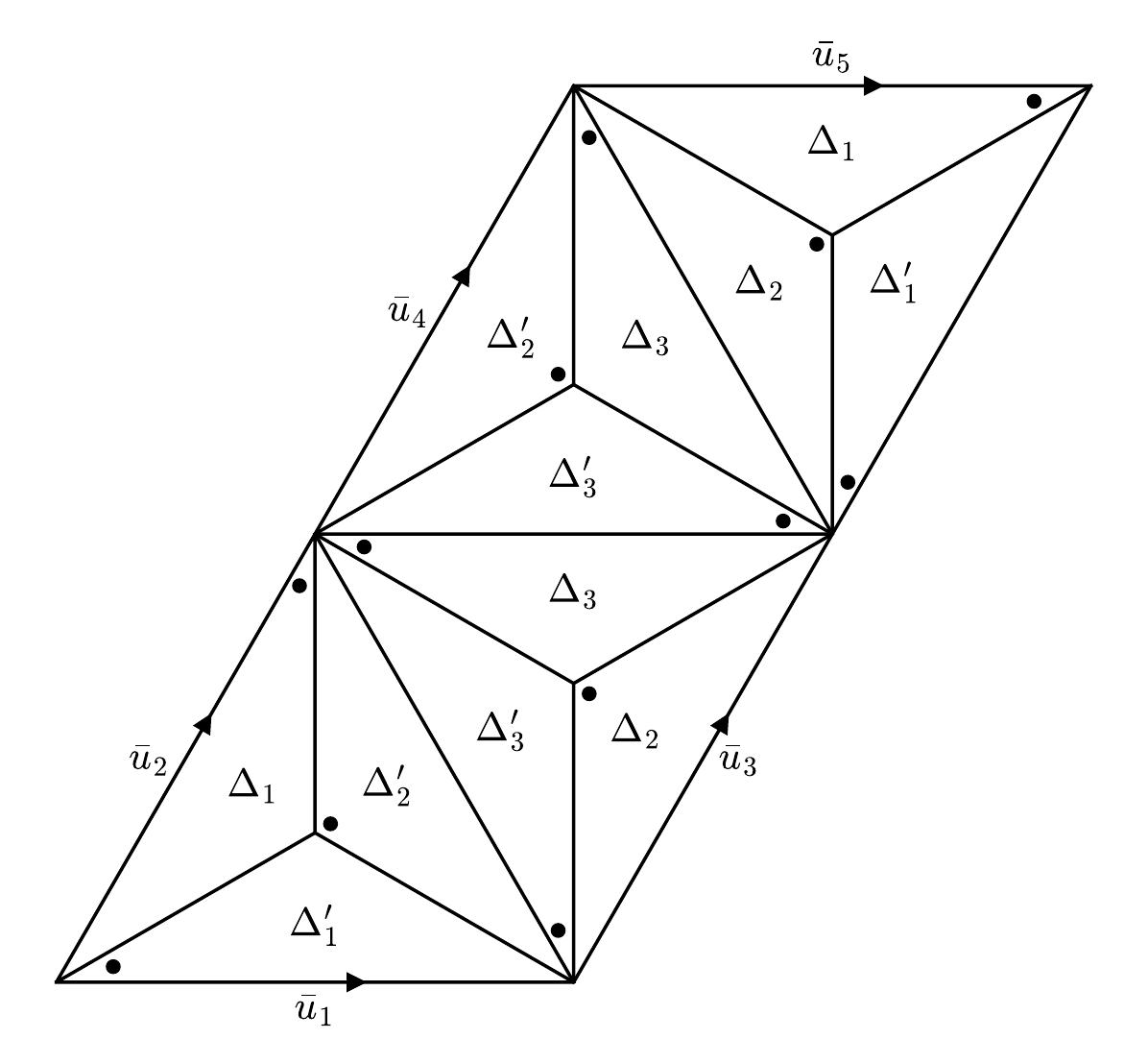} 
   \caption{The induced triangulation of the cross section of $c_1$.}
   \label{fig: Cusp1}
\end{figure}

A fundamental domain for the action of $P$ on $\widetilde{W}$ consists of a union of two adjacent tetrahedra $\widetilde{\Delta}_i$ and $\widetilde{\Delta}_i'$ in each layer, and it follows that $M$ has an ideal triangulation made up of six ideal tetrahedra. Let $\Delta_i$ and $\Delta_i'$ be the images of $\widetilde{\Delta}_i$ and $\widetilde{\Delta}_i'$ in $M$ and classify their edges as top, bottom, or side edges according to their lifts.  

We pause in our discussion momentarily to notice that, with this construction, it is easy to identify certain useful involutions of $M$.  Consider the group generated by the order-2 Euclidean rotations of $\mathbb{R}^2$ around the points of $\frac{1}{2}\mathbb{Z}^2$. The elements of this group preserve the integer lattice and the slopes of lines, so they induce an action on $\widetilde{W}$. They normalize $P$ and thus induce an action on $W=\widetilde{W}/P$, which descends to the further quotient $M$ because slopes are preserved.  For instance, if we rotate about the midpoint of a lift of a side edge of $\Delta_1$ we obtain a simplicial involution $\sigma_1 \co M \to M$ which interchanges $\Delta_1$ and $\Delta'_1$.  

Similarly, we have a simpicial involution $\sigma_2 \co M \to M$ which inverts the top edges of $\Delta_3$ and $\Delta'_3$.  These edges all have slope $0$ and are side edges of the remaining four tetrahedra.  It follows that 
\begin{align*}
\sigma_2(\Delta_3)&= \Delta_3 & \sigma_2(\Delta_j)&= \Delta'_j,
\end{align*}
where $j=1,2$.  By considering the combinatorics of the triangulation, we also conclude that $\sigma_2$ interchanges the cusps of $M$.

Returning to our discussion of $\mathcal{D}_0(M)$, we assign indeterminate edge parameters $z_i$ and $z_i'$ to the respective top edges of the tetrahedra $\Delta_i$ and $\Delta_i'$.  Assign parameters to all remaining edges using the usual conventions and the functions $\zeta_1$ and $\zeta_2$.  After noticing that the folding quotient identifies all edges of slopes $0$ and $1$ in $\Sigma_1$ and all edges of slopes $1/3$ and $2/7$ in $\Sigma_4$, we see that the following collection of edge equations define $\mathcal{D}_0(M)$.
\begin{align*}
(1/2) && 1&=z_2 \, \zeta_2(z_1) \,\zeta_2(z_1') \\
(1/2) && 1 &= z_2' \, \zeta_2(z_1) \, \zeta_2(z_1') \\
(1/4) && 1&=z_{2} \,\zeta_1(z_{3}) \,\zeta_1(z_{3}') \\
(1/4) && 1 &= z_{2}' \,\zeta_1(z_{3}) \,\zeta_1(z_{3}')\\
(0:1) &&  1&= z_1 \, z_1' \, z_3 \, z_3' \, \left( \zeta_1(z_1) \, \zeta_1(z_1')\right)^2 \,  \left( \zeta_1(z_2) \, \zeta_1(z_2')\right)^2 \\
(1/3:2/7):&&1&= z_1 \, z_1' \, z_3 \, z_3' \,  \left( \zeta_2(z_2) \, \zeta_2(z_2')\right)^2 \, \left( \zeta_2(z_3) \, \zeta_2(z_3')\right)^2.
\end{align*}
We have labeled each of these equations with the slope(s) of the edge from which it arises.  We note that this triangulation has two edges of slope $1/2$, one a top edge of $\Delta_2$ and one of $\Delta_2'$.  Similarly for the edges of slope $1/4$.  

By the first pair of equations, we have
\begin{align*}\label{zee two}
 	z_2=z_2'=\frac{z_{1}' z_{1}}{{\left(z_{1}' - 1\right)} {\left(z_{1} - 1\right)}}\end{align*}
and we can eliminate the parameters $z_2$ and $z'_2$ from the equations that follow.  This also means that the second equation in the second pair is redundant.  Finally, the product of the right hand sides of the last four equations is one.  Therefore, a point of $\mathcal{D}_0(r_3)$ is determined by its coordinates $(z_1, z_1', z_3, z_3')$ and  
 \begin{align*}
(1/4) && 1&= \frac{z_1' z_1}{(z_1 - 1)(z_1' - 1)(z_3 - 1)(z_3' - 1)}\\
 (0:1) &&1&= \frac{z_1 z_1' z_3 z_3' (z_1-1)^2 (z_1'-1)^2}{(z_1+z_1'-1)^4}.
\end{align*}
These equations are obtained from those above by substituting for $z_2$ and $z_2'$.

The combinatorics of the induced triangulation of the cross section of one of the cusps is indicated in Figure \ref{fig: Cusp1}.  We will refer to this cusp as $c_1$ and to the other as $c_2$.  In the figure, one of the ideal vertices corresponding to the cusp is placed at $\infty \in \bound \H^3$ and the horizontal face of $\Delta_1'$ is taken to be the ideal triangle with vertices $\{ 0, 1, \infty\}$.  The edge parameters of the tetrahedra and the gluing pattern of the Sakuma-Weeks triangulation determine the remainder of the image. The convention in the figure is that the dots are placed in the corners whose vertical edges are labeled $z_j$ or $z_j'$. 

Identifying opposite sides of the parallelogram in the Figure yields a cross section $T_1$ of $c_1$, and we obtain a meridian-longitude pair of curves generating $\pi_1 T_1$ by taking $\fm$ to be the horizontal side (a single edge) and $\fl$ the diagonal side. We are interested in the locus of $\bz\in\mathcal{D}_0(M)$ where $c_1$ is complete, which is the curve determined by setting $\mu(\fm)(\bz) = 1$.  From Equation (\ref{mu fun}) and Figure \ref{fig: Cusp1}, we get
\[ \mu(\fm)(\bz) = - \zeta_1(z_2')z_3\zeta_1(z_2)\zeta_1(z_1)z_1\zeta_1(z_1') = -\frac{z_1(1-z_1)(1-z_1')z_3}{(1-z_1-z_1')^2}. \]

Requiring that  $\mu(\fm)(\bz)=1$ results in a formula for $z_3$ in terms of $z_1$ and $z_1'$ and, together with equation $(0:1)$, we get a similar formula for $z_3'$.   In particular,
\begin{align*}
 & z_3 = -\frac{(1-z_1-z_1')^2}{z_1(1-z_1)(1-z_1')} && z_3' = -\frac{(1-z_1-z_1')^2}{z_1'(1-z_1)(1-z_1')}. \end{align*}
After plugging these into equation $(1/4)$, clearing denominators, and simplifying, we find that the locus $\{\bz\in\mathcal{D}_0(M)\,|\,\mu(\fm)(\bz) = 1\}$ is determined by the set of $(z_1,z_1')\in (\mathbb{C}-\{0,1\})^2$ satisfying 
\[0= (z_1'^3 z_1 + z_1'^{2} z_1^{2} + z_1' z_1^3 - z_1'^2 - 3   z_1' z_1 - z_1^2 + 2   z_1' + 2   z_1 - 1)\cdot(z_1' + z_1 - 1).\]
 


\begin{lemma}\label{inance} For $M$, triangulated and with cusp cross section $T_1$ and its meridian $\fm$ chosen as above, the curve component $V_0$ of $\{\bz\in\mathcal{D}_0(M)\,|\,\mu(\fm)(\bz) = 1\}$ containing the complete structure $\bz^0$ is isomorphic to the complex affine variety determined by
\[ p(z_1,z_1') = z_1'^3 z_1 + z_1'^{2} z_1^{2} + z_1' z_1^3 - z_1'^2 - 3   z_1' z_1 - z_1^2 + 2   z_1' + 2   z_1 - 1. \]
\end{lemma}

\begin{proof} We first observe that $p$ is irreducible over $\C$. Let $q$ be the polynomial obtained by setting $z_1'=x+1$ and $z_1=y$.  Then
 \[ q(x,y) = x^3 y + x^2 y^2 + x y^3 + 3  x^2 y + 2  x y^2 + y^3 - x^2\]
 and $p$ is irreducible over if and only if $q$ is. By the irreducibility criterion given in \cite{G}, it is enough to show that the Newton polygon $N$ for $q$ is integrally indecomposable.  Following the approach given therein, we consider the sequence 
 \[ \{ (1,1), (-1,1), (-1,1), (-1,0), (2,-3) \}\]
of vectors corresponding to the edges of $N$.  Since there is no proper subsequence which sums to the zero vector, $N$ is integrally indecomposable, so $q$ and $p$ are irreducible.

As observed above the Lemma, the locus $\{\bz\in\mathcal{D}_0(M)\,|\,\mu(\fm)(\bz) = 1\}$ contains the graph of $(z_2,z_2',z_3,z_3')$, given as functions of $(z_1,z_1')$ by the formulas above, over the locus $p=0$. 

It remains to show that $\bz^0$ does not satisfy $z_1'+z_1 - 1 =0$.  Recall, that the involution $\sigma_1$ interchanges $\Delta_1$ and $\Delta'_1$.  By Mostow rigidity, $\sigma_1$ is represented by an isometry of the complete hyperbolic structure on $M$.  Hence, it must be true that at $\bz^0$, $z_1 = z_1'$.  Since this cannot happen if $z_1'+z_1 - 1 =0$, our proof is complete.
\end{proof}

\begin{prop}\label{622shape}
With the hypothesis and notation of Lemma \ref{inance}, 
\[ \tau_{c_1} = \frac{z_1}{1-z'_1} + \frac{z'_1}{1-z_1} \]
is a cusp parameter function for $c_1$ on the curve $V_0$.
\end{prop}

\begin{proof}
Take $\tau_{c_1} = \tau(\fl)/\tau(\fm)$.  To compute $\tau(\fl)$ and $\tau(\fm)$, we refer to Figure \ref{fig: Cusp1} and Equation (\ref{tau}) from Proposition \ref{ratfun} and choose $\fm$ as the reference edge $f$.  Hence,
\begin{align*}
\tau(\fm)&= 1 & \tau(\fl)&=z'_1 \zeta_1(z_1) - z'_1 \zeta_1(z_1) z_1 \zeta_2(z'_2)^2 \zeta_2(z'_3)^2z_3.
\end{align*}
In the function field for $V_0$, the rational function $\tau(\fl)$ is equivalent to the function given in the statement of this proposition.
\end{proof}

\begin{cor}\label{nonconst_cusp} The cusp parameter function $\tau_{c_1}$ is non-constant on $V_0$.
\end{cor}

\proof Take $\tau_{c_1}$ as given in Proposition \ref{622shape}.  At $\bz^0$, cusp cross-sections have the structure of genuine Euclidean tori, so the translations corresponding to $\fm$ and $\fl$ are linearly independent over $\mathbb{R}$.  In particular, $\tau_{c_1}(\bz^0) \notin \R$.

On the other hand, if we take $z_1'=-1$, the polynomial $p$ specializes to the real cubic $-z_1^3+4z_1-4$. This has a real root $r \neq 1$.  It is clear that $\tau_{c_1}(-1,r) \in \R$, so $\tau_{c_1}$ is not constant on $V_0$. 
\endproof

\begin{cor}\label{eree}\CorEEE\end{cor}

\begin{proof} This is an immediate consequence of Corollary \ref{nonconst_cusp} and Corollary \ref{TwoCpt def}.   Since the simplicial involution $\sigma_2$ exchanges the cusps of $M$, Fact \ref{involution} promotes the result to both cusps.
\end{proof}

\section{Numerical examples}\label{two component}

Sections \ref{whitehead} and \ref{sixtwotwo} addressed link complements whose cusp shapes cover rigid Euclidean orbifolds and hence satisfy criterion (1) of Corollary \ref{TwoCpt}. This behavior is uncommon among two-component link complements.  Here we exhibit many links whose complements lack this property, and hence yield hyperbolic knot complements under Dehn surgery which generically lack hidden symmetries. 

We use the fact below to convert criterion (1) of Corollary \ref{TwoCpt} into a condition that can be checked by computer.   We will use SnapPy \cite{SnapPy} and Sage \cite{sagemath}.

\begin{fact}\label{cusp field}  If the shape of a cusp $c$ of a complete, finite-volume hyperbolic $3$-manifold covers a rigid Euclidean orbifold then the \mbox{\rm cusp field of $c$}, the smallest field containing its cusp parameter, is either $\mathbb{Q}(i)$ or $\mathbb{Q}(\sqrt{-3})$.\end{fact}

This follows from the fact that the Euclidean lattice uniformizing a horospherical cross section of such a cusp $c$ lies in a $(2,4,4)$-, $(3,3,3)$-, or $(2,3,6)$-triangle group.  

\subsection{Links with at most nine crossings}\label{9x} The  two-component links with at most $9$ crossings are tabulated in \cite[App.~C]{Rolfsen}. By inspection, each such link has at least one unknotted component. The resulting census is built in SnapPy, and we used the following sequence of commands in Sage to check the cusp field of each cusp of each hyperbolic link complement in it.

\begin{quote}  \texttt{sage: import snappy}\\
	\texttt{sage: for M in snappy.LinkExteriors(num\_cusps=2)[:90]:}\\
	\texttt{....: \quad N = M.high\_precision()}\\
	\texttt{....: \quad if N.volume() > 1:}\\
	\texttt{....: \qquad c0 = N.cusp\_info(0)}\\
	\texttt{....: \qquad c1 = N.cusp\_info(1)}\\
	\texttt{....: \qquad print (M, factor(c0.modulus.algdep(20)),}\\\color{white}.\color{black}
		\hspace{2.2in}\texttt{factor(c1.modulus.algdep(20)))}\\
	\texttt{....:}\end{quote}

Reading the above from the top, the first line calls SnapPy from within Sage.  (This is assuming that the two programs have been configured to work with each other, see the SnapPy website for instructions on how to accomplish this.)  The second line iterates over the pre-built table of triangulated complements of links up to 10 crossings in SnapPy; we have restricted to two-component links through 9 crossings. The third calls a high-precision version $N$ of each manifold $M$. The next line excludes the non-hyperbolic examples, which SnapPy computes to have $0$ or infinitesimal volume. (These are $4^2_1$, $6^2_1$, $7^2_7$, $8^2_1$, $9^2_{43}$, $9^2_{49}$, $9^2_{53}$ and $9^2_{61}$. Each can be directly shown to be non-hyperbolic.)

The next two lines pull information on the two cusps of $M$. In particular ``\texttt{c0.modulus}'' and ``\texttt{c1.modulus}'' are the shapes of these cusps relative to their ``geometrically preferred'' basis, consisting of the shortest and second-shortest translations. We finally print the notation from \cite[App.~C]{Rolfsen} for the link corresponding to $M$, together with factored polynomials satisfied by each of $z$ and $w$.  In particular, the command ``\texttt{z.algdep($n$)}'' finds a polynomial of degree at most $n$ that is (approximately) satisfied by the (approximate) algebraic number $z$.

We regard the routine above as having failed on a manifold $M$ if either of the two polynomials produced has degree $20$. When run with high precision as above, this only happens on a handful of examples (described below). Moreover, as a double check, we compare the above data to what is generated by the routine below:

\begin{quote}  \texttt{sage: import snappy}\\
	\texttt{sage: for M in snappy.LinkExteriors(num\_cusps=2)[:90]:}\\
	\texttt{....: \quad if M.volume() > 1:}\\
	\texttt{....: \qquad T = M.trace\_field\_gens()}\\
	\texttt{....: \qquad T.find\_field(prec=200, degree=20, optimize=True)}\end{quote}
	
The command ``\texttt{trace\_field\_gens()}'' produces a set of generators for the trace field of $M$ as ApproximateAlgebraicNumbers, a data structure that the Sage command ``\texttt{find\_field()}'' takes as input for an LLL-type algorithm (see eg.~\cite[\S 3]{CGHN}) that seeks a minimal polynomial for the field that they generate. The arguments ``prec'' and ``degree'' respectively specify the precision that the algorithm uses and the maximum degree polynomial it seeks. This routine as written above fails (ie.~terminates with no output) only on the link $9^2_{22}$, but for this link, increasing the precision to 300 produces a minimal polynomial of degree $12$, the largest observed among the two-component links up to nine crossings.

It is well known (and follows from the proof of Prop.~\ref{regfun}, see especially equation (\ref{see?})), that each cusp field of $M$ is contained in the trace field of $M$. Comparing degrees of the polynomials generated by the two routines above, we find that in fact equality holds in all cases for which the first routine does not fail. The links for which it does fail are $9^2_{18}$, $9^2_{22}$, $9^2_{25}$, $9^2_{34}$, $9^2_{35}$ and $9^2_{39}$, but after randomizing the triangulation of $S^3-9^2_{34}$ it succeeds and again gives equality.

Of the remaining links, the complements of the last two have trace fields of odd degree, so they cannot have quadratic cusp fields. And $S^3 - 9^2_{22}$ also has no quadratic cusp field since its trace field $K$, while of even degree, has no quadratic subfield as shown by the following Sage commands that produce an empty list $S$:

\begin{quote} \texttt{sage: K.<a> = NumberField(}\ [$p$]\ \texttt{)}\\
	\texttt{sage: S = K.subfields(degree=2)}\end{quote}

\noindent (Here ``$p$'' is the minimal polynomial for $K$ identified by \texttt{M.trace\_field\_gens()}.)

Unfortunately the same Sage commands, applied to the (shared) trace field of $S^3-9^2_{18}$ and $S^3-9^2_{25}$, show that it does have $\mathbb{Q}(i)$ as a subfield. For these we use \texttt{M.cusp\_translations()} (taking $M$ to be the complement of each in turn), which produces a list of explicit translations generating the lattices of the two cusps. For each cusp of each complement, we then experiment with changes of basis until we find a pair whose ratio can be handled by \texttt{algdep}. In all cases the degree of the resulting polynomial is $12$, matching that of the trace field.
	
After all this, out of the $82$ hyperbolic two-component links tested, only $5^2_1$, $6^2_2$, $7^2_8$, $8^2_{15}$, and $9^2_{47}$ have unknotted cusps with cusp field $\mathbb{Q}(i)$ or $\mathbb{Q}(\sqrt{-3})$.  Every other two-component link with at most nine crossings fails criterion (1) of Corollary \ref{TwoCpt}.  Hence, for these links, Dehn surgery on one component yields at most finitely many hyperbolic manifolds with hidden symmetries.

In fact, the complements of $7^2_8$, $8^2_{15}$, and $9^2_{47}$ are each isometric to that of the Whitehead link $5^2_1$. This can be verified with SnapPy, but it can also be seen directly.  The unknotted component of $7^2_8$ bounds a disk which intersects its other component, a trefoil knot, twice.  Cutting along this disk, rotating by multiples of $2\pi$, and regluing yields a sequence of isometries between these links' complements.

The complements of $5^2_1$ (hence also of $7^2_8$, $8^2_{15}$, and $9^2_{47}$) and $6^2_2$ do have rigid cusps and satisfy Corollary \ref{TwoCpt} (1). But these examples were considered individually earlier in this paper. We have proved:

\begin{thm_star}\label{data}\DataThm\end{thm_star}

\subsection{Non-AP knots}\label{non-AP}

As mentioned in the introduction, a knot $K$ is \textit{AP} if every closed incompressible surface $S\subset S^3-\mathcal{N}(K)$ contains an essential closed curve that bounds an annulus immersed in $S^3-\mathcal{N}(K)$ with its other boundary component on $\partial \mathcal{N}(K)$. While most low-complexity knots are AP, the class of \textit{non}-AP knots seems most likely to provide new examples with hidden symmetries \cite{BoiBoCWa2}.  Here we exhibit a family of non-AP knots $K_n \subset S^3$ whose complements generically lack hidden symmetries.  In particular, a knot is not AP if its complement contains a closed incompressible, anannular surface. The $K_n$ have this property.

We first construct the $K_n$. Consider the link $K \sqcup K' \subset S^3$, both components of which are unknotted. For $n\geq 3$ let $p_n \co S^3 \to S^3$ be the $n$-fold cyclic branched cover, branched over $K'$, that is an orbifold cover of the $(n,0)$-surgery on $K'$ with the standard framing. Since $K$ and $K'$ have linking number one, the preimage $K_n$ of $K$ is a knot in $S^3$ for each $n$. Let $M_n = S^3-K_n$.

For the embedded two-sphere $S_0$ indicated in Figure \ref{pmpy link}, define $S_n = p_n^{-1}(S_0)$.  It is easy to see that $S_n$ is connected, hence that $M_n \backslash S_n$ consists of two components.  Let $\widetilde{N}_n$ be the non-compact component and $\widetilde{N}'_n$ the other one.  The branched cover $p_n$ restricts to branched covers $\widetilde{N}_n \to N$ and $\widetilde{N}'_n \to N'$, where $N$ is the non-compact component of $\left(S^3-K\right) \backslash S_0$.


\begin{figure}[h]
\begin{tikzpicture}[scale=1.1]

%
%
%
%
%
%
%
%
%
%

\node at (.8,-1.4) {$K'$};
\node at (1.75,.35) {$K$};
\node at (-1.2,0) {$S_0$};

\draw [color=gray] (0,0) circle [radius=1];
\draw [thick] (0,-1) -- (0,-0.25);
\draw [thick] (0,0.25) -- (0,1);
\draw [thick] (0.7,-1.21) arc (-60:60:1.4);

\draw [thick] (0,-1) to [out=-90, in=210] (0.7,-1.21);
\draw [thick] (0,1) to [out=90, in=150]  (0.7,1.21);

\draw [thick] (1,0.1) to [out=0,in=90] (1.15,0) to [out=-90,in=0] (1,-0.1);
\fill [color=white] (1.06,-0.08) circle [radius = 0.05];

\draw [thick] (1.25,0.45) to [out=190,in=80] (1.06,0.13);
\fill [color=white] (1.4,0.15) circle [radius=0.1]; 
\draw [thick] (1.4,0.45) to [out=0, in=0] (1.4,0.15) to [out=180,in=170] (1.35,-0.15);
\fill [color=white] (1.35,-0.45) circle [radius=0.1];
\draw [thick] (1.45,-0.2) to [out=-10, in=0] (1.4,-0.45) to [out=180,in=-90] (1.055,0.045);

\draw [thick] (1,0.1) to [out=180,in=45] (0.3986,-0.2281);
\draw [thick] (1,-0.1) to [out=180,in=45] (0.4505,-0.3155);
\fill [color=white] (0,-0.5) circle [radius=0.1];
\draw [thick] (0.2581,-0.3686) arc (-55:-262:0.45);
\draw [thick] (0.3155,-0.4505) arc (-55:-264:0.55);

\draw [thick] (0.3182,0.3182) arc (45:82:0.45);
\draw [thick] (0.3889,0.3889) arc (45:84:0.55);
\draw [thick] (0.3182,0.3182) to [out=-45,in=120] (0.55,0);
\draw [thick] (0.3889,0.3889) to [out=-45,in=120] (0.65,0.05);

\draw [thick] (0.62,-0.25) to [out=-60,in=30] (0.55,-0.45) to [out=200,in=-60]  (0.4,-0.35) to [out=120,in=-30] (0.1,0.1) to [out=150,in=-90] (0,0.25);
\draw [thick] (0.72,-0.2) to [out=-60,in=30] (0.6,-0.55) to [out=210,in=-60] (0.3,-0.35) to [out=120,in=0] (0.1,-0.1) to [out=180,in=90] (0,-0.25);


\end{tikzpicture}
\caption{The link $K \sqcup K'$ and the 2-sphere $S_0$.}
\label{pmpy link}
\end{figure}
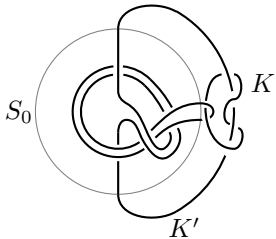

The compact component $\widetilde{N}_n'$ of $M_n\backslash S_n$ is the manifold $M_{n,1}$ of \cite{ZP} (see especially their Figure 5), so by the main result of that paper it admits a complete hyperbolic structure with finite volume and geodesic boundary. This implies that $S_n$ is incompressible and anannular in $\widetilde{N}_n$, see eg. \cite[p.~395]{CheDe}. 

Similarly, Theorem 1.1 of \cite{Frigerio_graphcomp} shows that $\widetilde{N}_n$ admits a complete hyperbolic structure with finite volume and geodesic boundary $S_n$. This is because $\widetilde{N}_n$ is the knot complement $M_g$ in a handlebody studied in \cite{Frigerio_graphcomp}, for $g= n-1$, as can be seen from Figure 1 there. The complement of an open regular neighborhood of the ``graph'' component $\Gamma^1_g$ there is a handlebody $N_g$ of genus $g+1$, and the pair $(N_g,\Gamma_g^0)$ has a rotational symmetry of order $g+1$ about a vertical axis in that Figure. The quotient of $N_g$ by this symmetry is a ball; and the union of the rotation's fixed locus with the ``knot'' component $\Gamma_g^0$ projects to a tangle in this ball isotopic to the intersection of $K\sqcup K'$ with the outside of $S_0$.

It thus follows that $S_n$ is incompressible and annanular in $\widetilde{N}'_n$, hence also in $M_n$.

\begin{prop_star}\label{pmpy}\PMPY\end{prop_star}

\begin{proof} For $n \geq 3$, take $K_n$ and $M_n$ as described above the Proposition.  As we've seen,  $K_n$ is not AP.  Moreover, Thurston's geometrization for Haken manifolds shows that each $M_n$ is hyperbolic.  Let $O_n$ be the hyperbolic orbifold with underlying space $S^3-K$ obtained by $(n,0)$-surgery on $K'$, and recall from above that $p_n$ restricts to an orbifold cover $M_n\to O_n$. We therefore conclude that $\mathrm{vol}\,M_n = n\cdot\mathrm{vol}\,O_n\to\infty$ as $n\to\infty$, as the volumes of the $O_n$ limit to that of $S^3-L$.

We finally claim that $M_n$ lacks hidden symmetries for all but finitely many $n$. Recall that the figure-eight is the only knot whose complement is arithmetic \cite{ReidFig8}. Because the figure-eight knot complement does not contain a closed essential surface,  every $M_n$ is non-arithmetic. Therefore, there is a unique minimal orbifold $Q_n$ in the commensurability class of $M_n$, and by \cite[Prop.~9.1]{NeumReid}, $Q_n$ has a rigid cusp for each $n$ such that $M_n$ has hidden symmetries. 

It follows that $O_n$ covers the rigid-cusped orbifold $Q_n$ for each $n$ such that $M_n$ has hidden symmetries. If this held for infinitely many $n$ then, by criterion (1) of Corollary \ref{TwoCpt}, the shape of the cusp of $S^3-L$ corresponding to $K$ would cover a rigid Euclidean orbifold. Computations with SnapPy and Sage show that this is not so. SnapPy identifies $L$ as the link L12n739 from the Hoste-Thistlethwaite table. Then running SnapPy within Sage as described in Section \ref{9x} shows that the trace field has degree 11, which rules out a cusp field of $\mathbb{Q}(i)$ or $\mathbb{Q}(\sqrt{-3})$.\end{proof}

\section{Other examples from the literature}\label{examples}

\begin{example}\label{2 3 n}  
Suppose that $n$ is odd and $L\subset S^3$ is the link pictured in Figure \ref{23n limit}.  This link is isotopic to the mirror image of the link $9_{59}^2$.  In particular, Theorem \ref{data} applies to $L$.  Since the $(-2,3,n)$-pretzel knot is obtained by $\frac{2}{n-1}$ Dehn filling on the unknotted component of $L$, at most finitely many $(-2,3,n)$ pretzel knot complements have hidden symmetries.

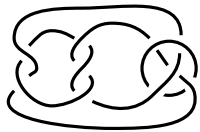
\begin{figure}[h]
\begin{tikzpicture}[scale=0.4]


\draw [very thick] (1.25,-0.25) to [out=45, in=-90] (1.5,0) to [out=90, in=-90] (0.75,1) to [out=90, in=180] (3,2) to [out=0, in=90] (6.3,1.1);
\draw [very thick] (1.25,0.75) to [out=45, in=180] (2,1.25) to [out=0, in=150] (2.75,1);
\draw [very thick] (1,0.25) to [out=-135, in=180] (2,-1.25) to [out=0,in=-45] (3.25,-0.25);
\draw [very thick] (2.75,-0.75) to [out=135, in=225] (3,0) to [out=45,in=-45] (3.25,0.75);
\draw [very thick] (2.75,0.25) to [out=135, in=225] (3,1) to [out=45, in=180] (4,1.5) to [out=0, in=135] (5.25,1);

\draw [very thick] (5.9,0) circle [radius=0.9];
\fill [color=white] (5.5,-0.7) circle [radius=0.3];
\fill [color=white] (6.7,-0.6) circle [radius=0.3];

\draw [very thick] (6.25,0.5) to [out=-90, in=45] (5.35,-0.9) to [out=-135, in=-25] (3.35,-1.1);
\draw [very thick] (5.5,0.6) to (5.85,0.1);
\draw [very thick] (0.75,-0.75) to [out=-135, in=180] (4,-2) to [out=0, in =-80] (6.75,-0.85) to [out=100,in=-45] (6.25,-0.25);

%
%
%
%
%

\end{tikzpicture}
\caption{This link is isotopic to the mirror image of the link $9_{59}^2$ from Rolfsen's tables.}
\label{23n limit}
\end{figure}

In fact, Macasieb--Mattmann \cite{MacMat} have shown that \textit{no} hyperbolic $(-2,3,n)$ pretzel knot complements have hidden symmetries, a stronger conclusion.  But our methods require only a SnapPy/Sage computation on $S^3-9^2_{59}$.  They also apply equally well to, for instance, the $(2,3,n)$-pretzel knots which are obtained from filling the unknotted component of a mirror image of $9^2_{57}$.\end{example}

\begin{example}\label{potential} 
Here we will combine Corollary \ref{TwoCpt} with computations of Aaber-Dunfield \cite{AabD} to show that at most finitely many knot complements with hidden symmetries can be obtained from the complement $W$ of the $(-2,3,8)$-pretzel link by Dehn filling the unknotted component.  We thank Nathan Dunfield for pointing out the relevance of the results of \cite{AabD} to our project.

We will apply Corollary \ref{TwoCpt}(2) since, as described in \cite{AabD}, $W$ is obtained by identifying faces of a single regular ideal octahedron so its cusps are square (covering $(2,4,4)$ triangle orbifolds). Let $c_1$ and $c_2$ be respective cusps corresponding to the the knotted and unknotted components of the $(-2,3,8)$-pretzel link.  In \cite{AabD}, they choose standard generators $l_j$ and $m_j$ for the peripheral subgroups of $\pi_1 W$.  They set $v_j = \log \mu(-l_j)$ and $u_j=\log \mu(m_j)$ and, in their proof of Theorem 5.5, show that the $v_j$ and $u_j$ are related by
\[ v_j(\bu) = \frac{1}{2}\, \frac{\partial \Phi}{\partial u_j}\]
where
\[ \Phi(u_1, u_2) = i\, (u_1^2+u_2^2) - \frac{3-i}{96} \, (u_1^4+u_2^4) - \frac{1+i}{16}\, (u_1^2u_2^2) + O\left(|\bu|^6\right). \]
The corresponding cusp parameter functions are given by $\tau_{c_j} = v_j/u_j$ so we obtain
\[ \tau_{c_1}(u_1, u_2) = i -   \frac{3-i}{48} \,  u_1^2  - \frac{1+i}{16} \, u_2^2 +  \hdots \]
The locus of points in $\mathcal{D}_0$ for which $c_1$ remains complete is $\left\{ \bz \, \mid \, u_1=0 \right\}$ and $\tau_{c_1}$ is non-constant on this curve.  Our conclusion therefore follows from Corollary \ref{TwoCpt def} (2) as claimed.

\end{example}

\begin{example}\label{protoBerge} Here we will apply Corollary \ref{TwoCpt} to (re-)prove that Dehn filling one cusp of the \textit{Berge manifold} $M = S^3 - L$, where $L$ is a certain two-component link in $S^3$ (see \cite[Fig.~1]{Hoffman_3comm}), produces at most finitely many hyperbolic knot complements with hidden symmetries. This was first established in the proof of \cite[Theorem 1.1]{Hoffman_3comm} by Hoffman, whose later result \cite[Theorem 6.1]{Hoffman_hidden} implies that in fact no hyperbolic knot complement with hidden symmetries is produced by Dehn filling a cusp of $M$. This stronger assertion is out of reach of Corollary \ref{TwoCpt def}.

$M$ is triangulated by four regular ideal tetrahedra (see eg.~\cite{FGGTV}). It therefore covers the Bianchi orbifold $\mathrm{PSL}(2,\mathcal{O}_3)$ and thus satisfies condition (1) of Corollary \ref{TwoCpt def}, so we will use condition (2). SnapPy finds an involution of $M$ exchanging the two cusps. As only one component of $L$ is unknotted, this involution does not extend to $S^3$; nor does it preserve the triangulation by regular ideal tetrahedra, since this induces different triangulations of the cusps' cross sections. Nonetheless, by appealing to Fact \ref{involution} we may conclude that neither cusp is geometrically isolated from the other after checking this for only one of them.

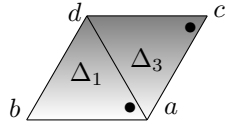
\begin{figure}
\begin{tikzpicture}[scale=0.8]
\begin{scope}[xscale=2, yscale=2]
\shade [](0,0)--(1,0)--(1.5,.865)--(.5,.865)--(0,0);
\draw (0,0)--(1,0);
\draw (0,0)--(.5,.865);
\draw(1,0)--(.5,.865);
\draw (1,0)--(1.5,.865);
\draw (.5,.865)--(1.5,.865);
\node at (.86,.1){$\bullet$};
\node at (1.36, .765){$\bullet$};
\node at (.5,.38){$\Delta_1$};
\node at (1,.5){$\Delta_3$};
\node at (1.2,.1){$a$};
\node at (-.1,.1){$b$};
\node at (1.6,.9){$c$};
\node at (.4,.9){$d$};
\end{scope}
\end{tikzpicture}
\caption{The induced triangulation of a cusp cross section}
\label{Berge cusp}
\end{figure}

We check the cusp $c$ corresponding to the unknotted component of $L$. A cross section inherits the triangulation pictured in Figure \ref{Berge cusp}, with parallel sides identified by translations. (The triangulation may be extracted from Regina \cite{regina} after entering the isometry signature ``jLLzzQQccdffihhiiqffofafoaa'' of $M$. This invariant of hyperbolic three-manifolds, introduced by Burton \cite{Burton}, can be computed by SnapPy.) Taking $\fm$ to be the projection of the horizontal sides of the parallelogram, equation (\ref{mu fun}) yields:
\[ \mu(\mathfrak{m}) = -\zeta_1(z_1)\zeta_1(z_3)z_3 = \frac{-z_3}{(1-z_1)(1-z_3)}. \]
The triangulation's edge equations are:
\begin{align}
	f_1 = 0,\ & \mbox{where}\ f_1(z_1,z_2,z_3,z_4) = z_4(1-z_3) - z_3 z_2(1-z_4)(1-z_1)\label{gluingberge1}\\
	f_2 = 0,\ & \mbox{where}\ f_2(z_1,z_2,z_3,z_4) =z_4(1-z_3) - z_1(1-z_2)\label{gluingberge2}
\end{align}
Equations (\ref{gluingberge1}) and (\ref{gluingberge2}) cut out $\mathcal{D}_0(M)$ in $\mathbb{C}^4$. To show that $c$ is not geometrically isolated from the other cusp of $M$, we need to show that the cusp parameter function $\tau_c$ is non-constant on the irreducible component $V_0$ that contains the discrete, faithful representation, of the algebraic subset of $\mathcal{D}_0(M)$ where $\mu(\fm) = 1$. By the above, this subset is cut out by
\[ f_3 = 0,\ \mbox{where}\ f_3(z_1,z_2,z_3,z_4) =z_3 + (1-z_1)(1-z_3) = 1- z_1(1-z_3). \]
To compute $\tau_c$ we take $\fl$ to be the projection of the diagonal sides of the parallelogram in Figure \ref{Berge cusp}, oriented from $b$ to $d$, and appeal to Proposition \ref{ratfun}. Letting the reference edge $f$ of the Proposition equal $\fm$, we obtain $\tau_c = \zeta_2(z_1)$, which is constant on $V_0$ if and only if $z_1$ is.

We could proceed as in Section \ref{sixtwotwo} to show that $z_1$ is non-constant on $V_0$, but for variety's sake we will use a simple calculus-based approach here.  Specifically, we will use the implicit function theorem to show that $z_2$ is a parameter for $V_0$ near the point $\vec{\eta} \doteq (\eta,\eta,\eta,\eta)$, for $\eta = \frac{1+\sqrt{-3}}{2}$, corresponding to the complete hyperbolic structure, and that $\frac{d^2z_1}{dz_2^2}\neq 0$ at $\vec{\eta}$. (As the given triangulation of the complete hyperbolic manifold $M$ is by regular ideal tetrahedra, $\vec{\eta}$, which we note is a fixed point of $\zeta_1$ and $\zeta_2$, corresponds to the complete structure. And since $\tau_c$ is an even function of $\log z_2$, we have $\frac{dz_1}{dz_2} = 0$.)

By the implicit function theorem, $z_2$ is a parameter for $V_0$ (and hence non-constant on it) in a neighborhood of $\vec{\eta}$ if and only if the partial derivative matrix \[M_2=
 \left(\frac{\partial f_i}{\partial z_j}(\vec{\eta})\right)_{j = 1,3,4}^{i=1,2,3} \]
is non-singular. A straightforward computation shows that it is. Implicit function theorem also implies that 
 $(\frac{d z_1}{d z_2}, \frac{d z_3}{d z_2}, \frac{d z_4}{d z_2})^{T}=-M_2^{-1} \vec{v}$
 around $\vec{\eta}$, where $\vec{v}= (\frac{\partial f_1}{\partial z_2}, \frac{\partial f_2}{\partial z_2}, \frac{\partial f_3}{\partial z_2})^{T}$.  Using this formula and chain rule repeatedly, one can write all higher power derivatives of $z_i$ with respect to $z_2$ in terms of all the $z_i$'s. From \cite{Mathematica} we see that, $\frac{d^2z_1}{dz_2^2}=\frac{i}{\sqrt{3}}\ne 0$ at $\vec{\eta}$ so $z_1$ is non-constant around $\vec{\eta}$ inside $V_0$, and by Corollary \ref{TwoCpt def}(2), $c$ is not isolated from the other cusp of $M$. As we observed above it now follows from Fact \ref{involution} that the other, knotted cusp of $M$ is also not isolated from $c$.\end{example}

\bibliographystyle{plain}
\bibliography{hidden_bib}

\begin{thebibliography}{10}

\bibitem{AabD}
John~W. Aaber and Nathan Dunfield.
\newblock Closed surface bundles of least volume.
\newblock {\em Algebr. Geom. Topol.}, 10(4):2315--2342, 2010.

\bibitem{Adams_toroidal}
Colin~C. Adams.
\newblock Toroidally alternating knots and links.
\newblock {\em Topology}, 33(2):353--369, 1994.

\bibitem{BBoCWaGT}
Michel Boileau, Steven Boyer, Radu Cebanu, and Genevieve~S. Walsh.
\newblock Knot commensurability and the {B}erge conjecture.
\newblock {\em Geom. Topol.}, 16(2):625--664, 2012.

\bibitem{BoiBoCWa2}
Michel Boileau, Steven Boyer, Radu Cebanu, and Genevieve~S. Walsh.
\newblock Knot complements, hidden symmetries and reflection orbifolds.
\newblock {\em Ann. Fac. Sci. Toulouse Math. (6)}, 24(5):1179--1201, 2015.

\bibitem{BP}
Michel Boileau and Joan Porti.
\newblock Geometrization of 3-orbifolds of cyclic type.
\newblock {\em Ast\'{e}risque}, (272):208, 2001.
\newblock Appendix A by Michael Heusener and Porti.

\bibitem{Burton}
Benjamin~A. Burton.
\newblock The {P}achner graph and the simplification of 3-sphere
  triangulations.
\newblock In {\em Computational geometry ({SCG}'11)}, pages 153--162. ACM, New
  York, 2011.

\bibitem{regina}
Benjamin~A. Burton, Ryan Budney, William Pettersson, et~al.
\newblock Regina: Software for low-dimensional topology.
\newblock {\tt http://\allowbreak regina-normal.\allowbreak github.\allowbreak
  io/}, 1999--2017.

\bibitem{Calegari}
Danny Calegari.
\newblock Napoleon in isolation.
\newblock {\em Proc. Amer. Math. Soc.}, 129(10):3109--3119 (electronic), 2001.

\bibitem{Champ}
Abhijit Champanerkar.
\newblock A-polynomial and bloch invariants of hyperbolic 3-manifolds.
\newblock Preprint.

\bibitem{CheDe}
Eric Chesebro and Jason DeBlois.
\newblock Algebraic invariants, mutation, and commensurability of link
  complements.
\newblock {\em Pacific J. Math.}, 267(2):341--398, 2014.

\bibitem{CGHN}
David Coulson, Oliver~A. Goodman, Craig~D. Hodgson, and Walter~D. Neumann.
\newblock Computing arithmetic invariants of 3-manifolds.
\newblock {\em Experiment. Math.}, 9(1):127--152, 2000.

\bibitem{SnapPy}
Marc Culler, Nathan~M. Dunfield, Matthias Goerner, and Jeffrey~R. Weeks.
\newblock Snap{P}y, a computer program for studying the geometry and topology
  of $3$-manifolds.
\newblock Available at \url{http://snappy.computop.org} (DD/MM/YYYY).

\bibitem{CuSh}
Marc Culler and Peter~B. Shalen.
\newblock Varieties of group representations and splittings of {$3$}-manifolds.
\newblock {\em Ann. of Math. (2)}, 117(1):109--146, 1983.

\bibitem{sagemath}
The~Sage Developers.
\newblock {\em {S}ageMath, the {S}age {M}athematics {S}oftware {S}ystem
  ({V}ersion 7.5.1)}, 2017.
\newblock {\tt http://www.sagemath.org}.

\bibitem{DM}
William~D. Dunbar and G.~Robert Meyerhoff.
\newblock Volumes of hyperbolic {$3$}-orbifolds.
\newblock {\em Indiana Univ. Math. J.}, 43(2):611--637, 1994.

\bibitem{KirbyList}
Rob~Kirby (Ed.).
\newblock Problems in low-dimensional topology.
\newblock In {\em Proceedings of Georgia Topology Conference, Part 2}, pages
  35--473. Press, 1995.

\bibitem{FGGTV}
Evgeny Fominykh, Stavros Garoufalidis, Matthias Goerner, Vladimir Tarkaev, and
  Andrei Vesnin.
\newblock A census of tetrahedral hyperbolic manifolds.
\newblock {\em Exp. Math.}, 25(4):466--481, 2016.

\bibitem{Frigerio_graphcomp}
R.~Frigerio.
\newblock An infinite family of hyperbolic graph complements in {$S^3$}.
\newblock {\em J. Knot Theory Ramifications}, 14(4):479--496, 2005.

\bibitem{G}
Shuhong Gao.
\newblock Absolute irreducibility of polynomials via {N}ewton polytopes.
\newblock {\em J. Algebra}, 237(2):501--520, 2001.

\bibitem{GL}
C.~McA. Gordon and J.~Luecke.
\newblock Knots are determined by their complements.
\newblock {\em J. Amer. Math. Soc.}, 2(2):371--415, 1989.

\bibitem{HLM}
Hugh~M. Hilden, Mar\'{\i}a~Teresa Lozano, and Jos\'{e}~Mar\'{\i}a
  Montesinos-Amilibia.
\newblock A characterization of arithmetic subgroups of {${\rm SL}(2,{\bf R})$}
  and {${\rm SL}(2,{\bf C})$}.
\newblock {\em Math. Nachr.}, 159:245--270, 1992.

\bibitem{Hoffman_3comm}
Neil Hoffman.
\newblock Commensurability classes containing three knot complements.
\newblock {\em Algebr. Geom. Topol.}, 10(2):663--677, 2010.

\bibitem{HIKMOT}
Neil Hoffman, Kazuhiro Ichihara, Masahide Kashiwagi, Hidetoshi Masai, Shin'ichi
  Oishi, and Akitoshi Takayasu.
\newblock Verified computations for hyperbolic 3-manifolds.
\newblock {\em Exp. Math.}, 25(1):66--78, 2016.

\bibitem{HMW}
Neil Hoffman, Christian Millichap, and William Worden.
\newblock Symmetries and hidden symmetries of {$(\epsilon,d_L)$}-twisted knot
  complements.
\newblock Preprint. ar{X}iv:1909.10571, September 2019.

\bibitem{Hoffman}
Neil~R. Hoffman.
\newblock On knot complements that decompose into regular ideal dodecahedra.
\newblock {\em Geom. Dedicata}, 173:299--308, 2014.

\bibitem{Hoffman_hidden}
Neil~R. Hoffman.
\newblock Small knot complements, exceptional surgeries and hidden symmetries.
\newblock {\em Algebr. Geom. Topol.}, 14(6):3227--3258, 2014.

\bibitem{Mathematica}
Wolfram~Research{,} Inc.
\newblock Mathematica, {V}ersion 11.1.
\newblock Champaign, IL, 2017.

\bibitem{LuoSchTill}
Feng Luo, Saul Schleimer, and Stephan Tillmann.
\newblock Geodesic ideal triangulations exist virtually.
\newblock {\em Proc. Amer. Math. Soc.}, 136(7):2625--2630, 2008.

\bibitem{MacMat}
Melissa~L. Macasieb and Thomas~W. Mattman.
\newblock Commensurability classes of {$(-2,3,n)$} pretzel knot complements.
\newblock {\em Algebr. Geom. Topol.}, 8(3):1833--1853, 2008.

\bibitem{Mar}
A.~Marden.
\newblock {\em Outer circles}.
\newblock Cambridge University Press, Cambridge, 2007.
\newblock An introduction to hyperbolic 3-manifolds.

\bibitem{MartPet}
Bruno Martelli and Carlo Petronio.
\newblock Dehn filling of the ``magic'' 3-manifold.
\newblock {\em Comm. Anal. Geom.}, 14(5):969--1026, 2006.

\bibitem{Menasco}
W.~Menasco.
\newblock Closed incompressible surfaces in alternating knot and link
  complements.
\newblock {\em Topology}, 23(1):37--44, 1984.

\bibitem{M}
Robert Meyerhoff.
\newblock The cusped hyperbolic {$3$}-orbifold of minimum volume.
\newblock {\em Bull. Amer. Math. Soc. (N.S.)}, 13(2):154--156, 1985.

\bibitem{Millichap}
Christian Millichap.
\newblock Mutations and short geodesics in hyperbolic 3-manifolds.
\newblock {\em Comm. Anal. Geom.}, 25(3):625--683, 2017.

\bibitem{MilliWord}
Christian Millichap and William Worden.
\newblock Hidden symmetries and commensurability of 2-bridge link complements.
\newblock {\em Pacific J. Math.}, 285(2):453--484, 2016.

\bibitem{Moser}
Harriet Moser.
\newblock Proving a manifold to be hyperbolic once it has been approximated to
  be so.
\newblock {\em Algebr. Geom. Topol.}, 9(1):103--133, 2009.

\bibitem{NeumReid}
Walter~D. Neumann and Alan~W. Reid.
\newblock Arithmetic of hyperbolic manifolds.
\newblock In {\em Topology '90 ({C}olumbus, {OH}, 1990)}, volume~1 of {\em Ohio
  State Univ. Math. Res. Inst. Publ.}, pages 273--310. de Gruyter, Berlin,
  1992.

\bibitem{NR}
Walter~D. Neumann and Alan~W. Reid.
\newblock Rigidity of cusps in deformations of hyperbolic {$3$}-orbifolds.
\newblock {\em Math. Ann.}, 295(2):223--237, 1993.

\bibitem{NZ}
Walter~D. Neumann and Don Zagier.
\newblock Volumes of hyperbolic three-manifolds.
\newblock {\em Topology}, 24(3):307--332, 1985.

\bibitem{ZP}
Luisa Paoluzzi and Bruno Zimmermann.
\newblock On a class of hyperbolic {$3$}-manifolds and groups with one defining
  relation.
\newblock {\em Geom. Dedicata}, 60(2):113--123, 1996.

\bibitem{PePor}
Carlo Petronio and Joan Porti.
\newblock Negatively oriented ideal triangulations and a proof of {T}hurston's
  hyperbolic {D}ehn filling theorem.
\newblock {\em Expo. Math.}, 18(1):1--35, 2000.

\bibitem{ReidFig8}
Alan~W. Reid.
\newblock Arithmeticity of knot complements.
\newblock {\em J. London Math. Soc. (2)}, 43(1):171--184, 1991.

\bibitem{ReidWalsh}
Alan~W. Reid and Genevieve~S. Walsh.
\newblock Commensurability classes of 2-bridge knot complements.
\newblock {\em Algebr. Geom. Topol.}, 8(2):1031--1057, 2008.

\bibitem{Rolfsen}
Dale Rolfsen.
\newblock {\em Knots and links}, volume~7 of {\em Mathematics Lecture Series}.
\newblock Publish or Perish, Inc., Houston, TX, 1990.
\newblock Corrected reprint of the 1976 original.

\bibitem{SW}
Makoto Sakuma and Jeffrey Weeks.
\newblock Examples of canonical decompositions of hyperbolic link complements.
\newblock {\em Japan. J. Math. (N.S.)}, 21(2):393--439, 1995.

\bibitem{Sh}
Peter~B. Shalen.
\newblock Representations of 3-manifold groups.
\newblock In {\em Handbook of geometric topology}, pages 955--1044.
  North-Holland, Amsterdam, 2002.

\bibitem{Th_notes}
W.~P. Thurston.
\newblock The geometry and topology of 3-manifolds.
\newblock mimeographed lecture notes, 1979.

\end{thebibliography}

\end{document}